\documentclass[11pt,a4paper,oneside]{article}

\usepackage{amsmath} 
\allowdisplaybreaks[1]  

\numberwithin{equation}{section} 
\usepackage{algorithm}
\usepackage{algorithmic} 
\usepackage{fancyhdr} 
\usepackage{graphicx} 
\usepackage{lipsum}  
\usepackage{setspace} 
\usepackage[]{fncychap} 
\usepackage[hyphens]{url} 
\usepackage{xcolor} 
\usepackage{tabularx}
\usepackage{appendix} 

\usepackage[english.es-tabla]{babel}
\usepackage{enumerate}
\usepackage{graphicx}
\usepackage[mathscr]{euscript}
\usepackage{amsmath,amssymb,amsthm}
\usepackage{amsmath,amsthm}
\usepackage{enumerate}
\usepackage{enumitem}
\usepackage{mathtools}
\usepackage[T1]{fontenc}
\usepackage[utf8]{inputenc}
\usepackage{mathrsfs}
\usepackage{subfig}
 
\usepackage{tikz-cd} 
\usepackage[nodayofweek]{datetime}
\usepackage{csquotes}

\newdateformat{monthyear}{\monthname[\THEMONTH], \THEYEAR}

\newtheorem{thm}{Theorem}[section]
\newtheorem{claim}[thm]{Claim}
\newtheorem{cor}[thm]{Corollary}
\newtheorem{lem}[thm]{Lemma}

\newtheorem{defn}[thm]{Definition}

\newtheorem{assump}[thm]{Assumption}
\newtheorem{rem}[thm]{Remark}

\newlist{steps}{enumerate}{1}
\setlist[steps, 1]{label = Step \arabic':}
\usepackage[a4paper,left=28mm, right=28mm, top=25mm, bottom=25mm]{geometry}

\usepackage[hidelinks]{hyperref}

\begin{document}

\title{A higher order sparse grid combination technique}

\author{Julia Mu{\~n}oz-Ech{\'a}niz\thanks{Mathematical Institute, University of Oxford, United Kingdom, Email: juliamunozechanizge@gmail.com}
\and Christoph Reisinger\thanks{Mathematical Institute, University of Oxford, United Kingdom, Email: christoph.reisinger@maths.ox.ac.uk}}
\date{\today} 

\maketitle

\begin{abstract}
We show that a generalised sparse grid combination technique which combines multivariate extrapolation of finite
difference solutions with the standard combination formula lifts a second order accurate scheme on regular meshes
to a fourth order combined sparse grid solution. In the analysis, working in a general dimension, we characterise all terms in a multivariate error expansion of the scheme as solutions of a sequence of semi-discrete problems. This is first carried out formally under suitable assumptions on the truncation error of the scheme, stability and regularity of solutions. We then verify the assumptions on the example of the Poisson problem with smooth data, and illustrate the practical convergence in up to seven dimensions.
\end{abstract}



\section{Introduction}\label{sec:intro}

The solution of high‐dimensional PDEs has been a long-running challenge to computational methods.
The traditional mesh-based approach is susceptible to an exponential growth with respect to the dimension of the degrees of freedom required for a desired accuracy.
This is seen most explicitly  for standard tensor‐product discretizations.  To mitigate this curse of dimensionality, several approaches are being utilised, including adaptive tensor approximations (see \cite{bachmayr2023low} for a survey), artificial neural networks 
(including deep Galerkin \cite{sirignano2018dgm}, PINNs \cite{raissi2019physics}, and deep BSDE methods \cite{han2018solving}),
and sparse grids (see \cite{Zenger1990} for the seminal paper; \cite{BungartzGriebel2004} for a survey, including sparse finite elements \cite{Bungartz1998, SchwabEtAl2008} or wavelets \cite{vonPetersdorffSchwab2004}, and the combination technique \cite{GriebelSchneiderZenger1992, hegland2016recent}), the latter of which are the focus of this work.

A sparse basis as introduced by \cite{Zenger1990} is understood to be a subspace of a standard basis on a tensor product mesh, chosen for an optimal trade-off between accuracy and computational complexity in spaces with mixed regularity.
It can be shown (see e.g.\ \cite{BungartzGriebel2004}) that such a space with standard multi-linear basis functions 
achieves an approximation accuracy of \( \mathcal{O}\bigl(N^{-2}(\log N)^{d-1}\bigr)\) for a certain regularity class, at a cost \( \mathcal{O}\bigl(N\,(\log N)^{d-1}\bigr)\), which is much reduced to that of a standard approach, with cost \(\mathcal{O}(N^d)\), for error \(\mathcal{O}(N^{-2})\).
Using standard Galerkin arguments, this approximation rate in suitable spaces translates into error bounds for finite element solutions \cite{Bungartz1992, Bungartz1998}.

In contrast, the sparse grid combination technique introduced in \cite{GriebelSchneiderZenger1992, hegland2016recent} defines solutions on each of the regular, anisotropic tensor product meshes that constitute the sparse grid basis, and then forms a linear combination. This vastly simplifies the implementation in a number of aspects: standard code on tensor meshes (e.g., finite differences, finite elements) can be used; the resulting discretisation matrices have a simpler banded structure than the sparse finite element stiffness matrix; the schemes on the different meshes can be solved fully in parallel.

\bigskip
The analysis of the combination technique, however, is more complicated and relies on an intricate multivariate expansion
of certain hierarchical increments.
There are broadly two classes of proofs. Those such as \cite{Pflaum1997} and \cite{PflaumZhou1999} work with Sobolev
space techniques for elliptic operators to derive directly bounds on the multivariate hierarchical surplus (see definition below), from which
the overall error follows directly. This was extended to more general operators in tensor spaces in \cite{griebel2014convergence}.
The second class of proofs first shows a specific multivariate expansion for the error on anisotropic Cartesian meshes, grouped by all possible subsets of coordinate directions, and then exploits an extrapolation mechanism to derive the overall sparse grid error.
This is the approach taken in \cite{BungartzGriebelRoeschkeZenger1994} via Fourier analysis and in \cite{reisinger2013analysis} by using a sequence of auxiliary semi-discrete problems to define the terms in the expansion.

\bigskip

Given that the sparse grid combination technique can be usefully viewed as an extrapolation formula based on a certain multivariate error expansion,
it is natural to ask whether the weights can be chosen, possibly by inclusion of more meshes, so that the order of convergence is increased from that of the schemes used on the individual meshes.
This is straightforward for regular (`full') tensor meshes, where linear combination of a coarse mesh and one refined in all directions can be chosen to eliminate the leading order error term, a method known as Richardson extrapolation
(see \cite{richardson1927viii}).
The direct application to sparse grid solutions at subsequent refinement levels fails due to the logarithmic factors in the leading order error terms, so that na{\"i}ve Richardson extrapolation of the combination solution does not improve the order.

In our analysis, instead, we carry out a multivariate extrapolation of each of the solutions on the regular meshes that contribute to the sparse grid, to obtain a solution on those meshes with an appropriate higher order multivariate error expansion. 
Then the sparse grid combination technique can be applied to give a sparse grid solution of higher order.
Put simply, instead of "combine, then extrapolate", we need to "extrapolate, then combine".

The analysis of such a scheme requires broadly two steps. First, we have to show that the low order solution on the tensor meshes has an error expansion which can be split into low order terms (which are to be extrapolated) and pure high order terms.
For this expansion, we are inspired by the recursive application of semi-discrete schemes in \cite{reisinger2013analysis}, but require substantially new steps to separate out all low order terms from the desired higher order terms.
Second, we have to find extrapolation weights which cancel the low order terms.

The main contributions of this paper are:
\begin{itemize}
    \item 
    the derivation of a fourth order sparse grid combination formula from regular tensor product solutions with a multivariate
    error expansion of second order;
    \item
    the proof of the required error expansion under generic conditions on stability, smoothness, and the truncation error of discrete and semi-discrete problems;
    \item
    the verification of the necessary assumptions on the example of the Poisson problem with smooth data;
    \item
    numerical confirmation of the predicted asymptotic error for Poisson problems in two to seven dimensions.
\end{itemize}

The remainder of this paper is structured as follows. Section \ref{Sec: Chapter 1} introduces the combination technique, states relevant existing results, and presents the main theorems of the paper.
Section \ref{sec:Chapter 4} derives the extrapolation formula and gives a proof that under a certain multivariate second order error expansion, the extrapolated solution has an analogous fourth order expansion, which ultimately leads to a fourth order error of the sparse grid solution. Section \ref{sec:Chapter 5} proves the form of the second order expansion under formal assumptions, which is the most technical part of this paper. Finally, Section \ref{sec:appl} verifies these formal assumptions for the central difference scheme for the Poisson problem and presents numerical tests.


\section{Background and Main Results}
\label{Sec: Chapter 1}

\subsection{Combination Technique}

The combination technique constructs a single global approximation by linearly combining solutions on a carefully chosen family of anisotropic tensor–product grids. 

Let $I^d = [0,1]^d$ denote the $d$-dimensional unit cube, and consider a family of Cartesian grids with mesh sizes $h_i = 2^{-l_i}$ in each direction $i = 1, \ldots, d$, with $l_i \in \mathbb{N}_0$. For a vector $\mathbf{l} = (l_1, \ldots, l_d) \in \mathbb{N}_0^d$, define $U(\mathbf{l}) := u_{2^{-\mathbf{l}}} = u_{\mathbf{h}}$ as the approximation of a function $u: I^d \rightarrow \mathbb{R}$ on the corresponding tensor-product grid with mesh sizes $\mathbf{h} = 2^{-\mathbf{l}}$.

In particular, we denote by $\mathbf{u}_{\mathbf{h}}$, the discrete vector of (approximated) values at the grid points, which is extended to $I^d$ by a suitable interpolation operator $\mathcal{I}$ as $u_{\mathbf{h}} := \mathcal{I} \mathbf{u}_{\mathbf{h}}$. We will ultimately be concerned with a setting where $u_{\mathbf{h}}$ is the finite difference solution to some PDE, in which case $u_{\mathbf{h}}$ is different from the vector obtained by evaluating the exact solution to the PDE, $u$, at the grid points. Therefore we denote the latter by $u(\mathbf{x}_{\mathbf{h}})$.

Given a family $U = (U(\mathbf{l}))_{\mathbf{l} \in \mathbb{N}_0^d}$ of such approximations, the \emph{hierarchical surplus} is defined as the sequence
\[
\delta U := \delta_1 \delta_2 \cdots \delta_d U,
\]
where the univariate difference operator $\delta_k$ acts component-wise:
\[
\delta_k U(\mathbf{l}) := 
\begin{cases}
U(\mathbf{l}) - U(\mathbf{l} - \mathbf{e}_k), & \text{if } l_k > 0, \\
U(\mathbf{l}), & \text{if } l_k = 0,
\end{cases}
\]
and $\mathbf{e}_k$ is the $k$-th unit vector in $\mathbb{R}^d$. The operators $\delta_k$ commute.

Given an index set $\mathcal{M}_n \subset \mathbb{N}_0^d$, the approximation at level $n$ is defined as:
\[
u_n := \sum_{\mathbf{l} \in \mathcal{M}_n} \delta U(\mathbf{l}).
\]
For the standard Smolyak index set $\mathcal{M}_n := \{\mathbf{l} \in \mathbb{N}_0^d : |\mathbf{l}|_1 := l_1 + \cdots + l_d \leq n\}$, the sparse grid approximation at level $n$ (see \cite{Smolyak1963}) takes the explicit form:
\begin{equation}
\label{combi_form}
u_n = \sum_{l = n}^{n + d - 1} a_{l-n} \sum_{|\mathbf{l}|_1 = l} U(\mathbf{l}), \qquad
a_i = (-1)^{d-1-i} \binom{d - 1}{i}, \quad 0 \leq i \leq d - 1.
\end{equation}

The main reason for using the combination technique is its ability to reduce computational cost compared to standard `full' meshes while maintaining a good level of accuracy, especially in high-dimensional problems. 
A full tensor product grid at level \( n \) in dimension \( d \) requires \( \mathcal{O}(2^{nd}) \) grid points, leading to an exponential growth in computational complexity with respect to \( d \).
In contrast, the combination technique employs a set of anisotropic grids indexed by multi-indices \( \mathbf{i} = (i_1, \ldots, i_d) \in \mathbb{N}_0^d \) satisfying \( |\mathbf{i}|_1 \leq n \). The number of such grids with $|\mathbf{i}|_1 = k$ is given by:
\[
\binom{k + d - 1}{d - 1} = O(k^{d-1}).
\]
Each individual grid at level \( \mathbf{i} \) has approximately \( O(2^{|\mathbf{i}|_1}) \leq O(2^{n}) \) grid points. Summing over all grids, the total number of degrees of freedom required by the combination technique satisfies:
\[
N_{\text{dof}} = \mathcal{O}(n^{d-1} 2^n).
\]
This leads to a significant reduction in computational cost compared to the full grid cost \( \mathcal{O}(2^{nd}) \), especially as the dimension \( d \) increases. The combination technique thus provides a tractable approach for high-dimensional problems where full grid computations would be infeasible, subject to error control, as discussed in the following.


\subsection{Pre-existing Error Analysis}

We briefly outline the approach to the error analysis of the combination technique that consists of two main components, as outlined in the introduction: first, establishing error expansions 
that decompose the discretisation error into contributions from lower-dimensional interactions; and second, demonstrating how the combination technique leverages these expansions to systematically cancel lower-order terms, leading to enhanced convergence rates and accuracy.

We consider linear PDEs on $I^d = [0, 1]^d$ and denote by $\partial_i u$ the partial derivative of $u$ with respect to $x_i$. For a multi-index $\alpha = (\alpha_1, \ldots, \alpha_d)$, let $D^{\alpha} u = \partial_1^{\alpha_1} \cdots \partial_d^{\alpha_d} u$. 





We now present the key result underlying the error analysis, exemplified by the Poisson problem, which shows that the total finite‐difference and interpolation error can be decomposed into a sum of lower‐dimensional contributions. 

\begin{thm}[Error Expansion of Finite Differences for Poisson problem]
\label{thm:error_expansion}
Let $u$ be a solution to the Poisson problem on $I^d$ such that for all $\alpha \in \{0, \ldots, 4\}^d$, $D^{\alpha} u$ is continuous, $\|D^{\alpha} u \|_{\infty} \leq K$ and $D^{\alpha} u\vert_{\partial I^d}=0$.

Let $\mathbf{u}_h$ denote the standard central finite difference approximation of order $2$ accuracy on a tensor product grid with mesh sizes $\mathbf{h} = (h_1, \ldots, h_d) = (2^{-l_1}, \ldots, 2^{-l_d})$,
and $u(\mathbf{x}_{\mathbf{h}})$ the exact solution evaluated on the same grid. Let $\mathcal{I}(\cdot)$ denote a multilinear interpolation operator on that mesh. Then:


\begin{enumerate}
    \item \textbf{Discretization Error Expansion} (\cite[Theorem 3.3]{reisinger2013analysis}) \\
    There exist bounded $w_{j_1, \ldots, j_m}(\mathbf{x}_h; h_{j_1}, \ldots, h_{j_m})$ such that 
    \begin{equation}
    \label{eq:error_expansion}
    u(\mathbf{x}_h) - \mathbf{u}_h = \sum_{m=1}^d \sum_{\{j_1, \ldots, j_m\} \subset \{1, \ldots, d\}} w_{j_1, \ldots, j_m}(\mathbf{x}_h; h_{j_1}, \ldots, h_{j_m}) \, h_{j_1}^2 \cdots h_{j_m}^2.
    \end{equation}

    \item \textbf{Interpolation Error of the Exact Solution} (\cite[Theorem 4.2]{reisinger2013analysis}) \\
    There exist $\alpha_{j_1, \ldots, j_m}(\mathbf{x}_h; h_{j_1}, \ldots, h_{j_m})$ such that 
    \[
    u(x) - (\mathcal{I}u(x_h))(x) = \sum_{m=1}^d \sum_{\{j_1, \ldots, j_m\} \subset \{1, \ldots, d\}} \alpha_{j_1, \ldots, j_m}(x; h_{j_1}, \ldots, h_{j_m}) \, h_{j_1}^2 \cdots h_{j_m}^2,
    \]
    for all $x\in I^d$, with
    \[
\|\alpha_{j_1, \ldots, j_m}(\cdot; h_{j_1}, \ldots, h_{j_m})\|_{\infty} \leq \left( \frac{4}{27} \right)^m \|\partial_{j_1}^2 \ldots \partial_{j_m}^2 u\|_{\infty}.
\]

    \item \textbf{Combined Error under Interpolation} (\cite[Theorem 4.3]{reisinger2013analysis}) \\
    As a consequence, the total error of the interpolated finite difference solution satisfies
    \begin{align}\label{general_error_expansion}
    u(x) - \mathcal{I}u_h(x) = \sum_{m=1}^d \sum_{\{j_1, \ldots, j_m\} \subset \{1, \ldots, d\}} v_{j_1, \ldots, j_m}(x; h_{j_1}, \ldots, h_{j_m}) \, h_{j_1}^2 \cdots h_{j_m}^2,
    \end{align}
    for all $x\in I^d$, where each $v_{j_1, \ldots, j_m}$ is uniformly bounded, and depends on both the discretization and interpolation errors.

\end{enumerate}
\end{thm}

Having established that the discretisation error admits an expansion into lower-dimensional contributions, we now turn to the second essential part of the analysis: understanding how the combination technique exploits this structure to improve accuracy. 
The next theorem gives a rigorous
bound for general order $p$ and dimension $d$, in the spirit of earlier results for $p=2$ and $d=2,3$ in \cite{GriebelSchneiderZenger1992}.

\begin{thm}[Sparse Grid Error Bounds; Theorem 5.4 in \cite{reisinger2013analysis}]
\label{thm:error_bounds}
Assume for all $u_h$ a pointwise error expansion of the form
\[
u - u_h = \sum_{m=1}^d \sum_{\{j_1, \ldots, j_m\} \subset \{1, \ldots, d\}} v_{j_1, \ldots, j_m}(\cdot; h_{j_1}, \ldots, h_{j_m}) \, h_{j_1}^p \cdots h_{j_m}^p,
\]
where $|v_{j_1, \ldots, j_m}| \leq K$ for all $1 \leq m \leq d$ and all $\{j_1, \ldots, j_m\} \subset \{1, \ldots, d\}$.

Then, the combination solution \eqref{combi_form}
fulfills the error estimate
\[
|u - u_n| \leq \frac{2K}{(d - 1)!} \left( \frac{2^p + 1}{2^{p-1}} \right)^{d-1} (n + 2(d - 1))^{d-1} \, 2^{-pn}.
\]
\end{thm}




An extension of this methodology to a fourth order monotone scheme with cubic spline interpolation is
considered in \cite{hendricks2017error}. Earlier experimental work on higher order sparse grid schemes applied to option pricing include \cite{leentvaar2006pricing} and \cite{hendricks2016high}.



Our approach differs crucially in the following aspect: rather than combining approximations from higher order schemes using the standard combination formula, we modify the weights in the combination technique to improve the order of accuracy for low order schemes. This allows for an implementationally trivial and computationally efficient modification of existing standard code.

The optimal choice of weights in an \emph{a priori} and \emph{a posteriori} sense is discussed in  
\cite{hegland2007combination}, including the `opticom' method, which finds the best approximation of a given function among linear combinations of projections onto subspaces.

More closely related to our framework below,
\cite{harding2014combination} assumes a so-termed order $(q,p)$ expansion of the form
\begin{equation}\label{expansion_harding}
u - u_{\mathbf{h}} = \sum_{m=1}^d \sum_{\substack{\{j_1, \ldots, j_m\} \\ \subset \{1, \ldots, d\} } } 
\eta_{j_1, \ldots, j_m} h_{j_1}^q \ldots h_{j_m}^q +
\gamma_{j_1, \ldots, j_m} (h_{j_1}, \ldots, h_{j_m}) h_{j_1}^p \ldots h_{j_m}^p,
\end{equation}
for some positive integers $p>q$,
and then solves a linear program to determine the optimal combination weights based on upper bounds on the expansion terms. 
Subsequently, \cite{harding2016adaptive} extended the results to different orders in different dimensions and embedded the extrapolation in an adaptive sparse grid combination technique.
Our analysis provides an expansion similar to \eqref{expansion_harding} in the case $q=2, p=4$, and can be seen
as a theoretical justification for the approach in \cite{harding2014combination, harding2016adaptive}.



\subsection{Main results of this paper}

   Let $\mathbf{u}_\mathbf{h}$ as before denote the finite difference approximation,
and
let \(u_{\mathbf{h}}\) denote the piecewise multilinear interpolant of $\mathbf{u}_\mathbf{h}$,
on the tensor‐product grid with mesh‐sizes \(\mathbf{h}=(h_1,\dots,h_d)\).


The combination technique is positioned in the context of multivariate extrapolation methods in \cite{bungartz1994extrapolation}.
In this work, we use multivariate extrapolation within the combination technique to increase the order.

Define an extrapolated solution
\begin{equation}\label{eq:extra}
\tilde{u}_{\mathbf{h}} = \sum_{k=0}^d \alpha_k \sum_{ \substack{ \{i_1, \ldots, i_k\} \\ \subset \{1, \ldots, d\} } } u_{\mathbf{h}}^{(i_1, \ldots, i_k)}    
\end{equation}
where $\alpha_k = \frac{(-4)^k}{(-3)^d}$, and, for $\mathbf{i} = (i_1, \ldots, i_k)$,
\[
\mathbf{u}_\mathbf{h}^{\mathbf{i}} \coloneqq \mathbf{u}_{\mathbf{h}^{\mathbf{i}}}, \quad
u_\mathbf{h}^{\mathbf{i}} \coloneqq u_{\mathbf{h}^{\mathbf{i}}},
\qquad \text{ where } \mathbf{h}^{\mathbf{i}} = (h_1^{\mathbf{i}}, \ldots, h_d^{\mathbf{i}}) \text{ with } 
h_j^{\mathbf{i}} = \begin{cases}
       h_j/2, & j \in \{i_1, \ldots, i_k\}, \\
       h_j, & \text{otherwise. }
   \end{cases}
\]
That is to say, the mesh width for the discrete solution $\mathbf{u}_\mathbf{h}^{\mathbf{i}}$
and interpolated solution ${u}_\mathbf{h}^{\mathbf{i}}$ have been refined by bisection in dimensions
$i_1, \ldots, i_k$.

\begin{rem}
    In dimension $d=1$, the extrapolation \eqref{eq:extra} is equivalent to Richardson extrapolation.
For $d\ge 2$, the extrapolation requires meshes with up to $d+1$ uni-directional refinements.
When we build a sparse combination solution below, which itself is a linear combination of solution on meshes from $d$ refinement levels, requires meshes from $2 d$ levels.

We can contrast this to the `splitting extrapolation' scheme of \cite{schuller1985efficient}, in two dimensions
\begin{equation}
u_{\mathrm{SL}} =  \frac{4}{3} u_{\mathbf{h}}^{(1)} + \frac{4}{3} u_{\mathbf{h}}^{(2)} -\frac{5}{3} u_{\mathbf{h}},
\end{equation}
which improves a second order error in $h_1$ and $h_2$ to order $h_1^4 + h_2^4 + h_1^2 h_2^2$.
This is of fourth order for a full mesh with $h_1 = h_2$, but the term $h_1^2 h_2^2$ prevents the higher order expansion
required for applying Theorem \ref{thm:error_bounds} with $p=4$.
\end{rem}

The following theorem establishes that the extrapolation formula \eqref{eq:extra}, which is to our knowledge novel\footnote{The formula is included in the PhD thesis \cite{Reisinger2004} (in German) of one of the authors, without the supporting convergence analysis provided here in Section \ref{sec:Chapter 5}.}, 
has the desired effect of canceling all terms that are purely of second order.

\begin{thm}\label{HOexp}

Assume there exist functions 
        \(
\beta_j\;\in\;L^\infty\bigl(I^d\times[0,1]^{d-1}\bigr),\;\gamma_{j_1,\dots,j_m}\;\in\;L^\infty\bigl(I^d\times[0,1]^m\bigr)\)
 such that the finite difference error satisfies the expansion
\begin{equation}\label{expansion_intro}
u - u_{\mathbf{h}} = \sum_{j=1}^d \beta_j(\cdot; h_1, \ldots, h_{j-1}, h_{j+1}, \ldots, h_d) h_j^2 + \sum_{m=1}^d \sum_{\substack{\{j_1, \ldots, j_m\} \\ \subset \{1, \ldots, d\} } } \gamma_{j_1, \ldots, j_m} (\cdot; h_{j_1}, \ldots, h_{j_m}) h_{j_1}^4 \ldots h_{j_m}^4
\end{equation}
for all $h_j$ that are negative powers of 2, and with $\| \gamma_{j_1, \ldots, j_m}\|_{\infty} \leq K$. 
Then 
$$
u - \tilde{u}_{\mathbf{h}} = \sum_{m=1}^d \sum_{\substack{\{j_1, \ldots, j_m\} \\ \subset \{1, \ldots , d\} } } \tilde{\gamma}_{j_1, \ldots, j_m}(\cdot; h_{j_1}, \ldots, h_{j_m}) h_{j_1}^4 \ldots h_{j_m}^4
$$
with $\| \tilde{\gamma}_{j_1, \ldots, j_m}\|_{\infty} \leq \frac{5^d}{3^d}K$. 
\end{thm}
\begin{proof}
    The proof is given in Section \ref{sec:Chapter 4}.
\end{proof}

\begin{rem}
We will argue below that any sufficiently smooth symmetric function of $h_1, \ldots, h_d$ which is 0 at the origin has
an expansion as on the right-hand side \eqref{expansion_intro}. A formal way to see this is to assume a
full Taylor expansion in $h_1^2,\ldots, h_d^2$ and first group all terms that contain a factor $h_j^2$ for at least one $j$ in the first sum, and then collect all remaining terms that are functions of $h_1^4,\ldots, h_d^4$ only, in the second sum.
\end{rem}

Theorem \ref{HOexp} brings us into the scope of Theorem \ref{thm:error_bounds} and allows us to define a higher order sparse grid solution.
Replacing $u_{\mathbf{h}}$ by $\tilde{u}_{\mathbf{h}}$ in the combination formula, and regrouping by $\mathbf{i}$,
we obtain the higher order combination solution at level $n$ as
\[
\hat{u}_n \coloneqq \sum_{l=n}^{n+d-1} a_{l-n} \sum_{\substack{|\mathbf{j}| = l \\\mathbf{h} = 2^{-\mathbf{j}} }} \tilde{u}_{\mathbf{h}}
=
\sum_{l= n}^{n+2d-1} \sum_{ \substack{ |\mathbf{i}| = l \\ \mathbf{h} = 2^{-\mathbf{i}}} } b_l u_{\mathbf{h}}
\]
with $b_l = \sum_{j = \max \{0, l-n-d\}}^{ \min \{l-n, d-1\} } a_j \alpha_{l-n-j} \binom{N_l}{l-n-j}$ and $\alpha_k = \frac{(-4)^k}{(-3)^d}$,  
    $N_l = \binom{m}{l} \binom{d-m}{k-l}$.
    A more detailed derivation is given in Section \ref{sec:Chapter 4}.

We can then write the error bounds for the higher order combination solution. Having established Theorem \ref{HOexp}, the following result is a direct consequence of Theorem \ref{thm:error_bounds}. 

\begin{cor}[Error bounds for the Higher Order Combination Solution]


Assume the setting of Theorem \ref{HOexp}, in particular that $u_{\mathbf{h}}$ satisfies \eqref{expansion_intro}.

Then \[|u - \hat{u}_n| \leq \frac{10}{3} \frac{K}{(d-1)!}\left( \frac{85}{24}\right)^{d-1} (n+2d-1)^{d-1} 2^{-4n}.\] 
\end{cor}

It remains to show that the premise \eqref{expansion_intro} is valid, which is the technically most involved part of the analysis. In Section \ref{sec:Chapter 5} we give a generic proof under certain natural assumptions on the truncation error of the discretisation scheme and of some semi-discrete auxiliary schemes (Assumption \ref{assumption}), in addition to assumptions on the smoothness and stability of solutions.
These have to be verified for specific applications. We do so for the Poisson problem in Section \ref{sec:appl}.

We have the following:



\begin{thm}\label{mainresult}
    
    
Under Assumptions \ref{newasumption} and \ref{assumption},
    there exist uniformly bounded functions $\beta_j$ and $\gamma_{j_1,\dots,j_m}$   
 such that the finite difference error satisfies, for all $h_j$ that are negative powers of 2, the expansion
    \begin{align}\label{HOexpansion}
    u(\mathbf{x}_h) - \mathbf{u}_h = \sum_{j=1}^d h_j^2 \beta_j(\mathbf{x}_h;h_1, \ldots, h_{j-1}, h_{j+1}, \ldots, h_d) + \sum_{m=1}^d \sum_{\substack{ \{ j_1, \ldots, j_m \} \\ \subset \{ 1, \ldots, d \} }} h_{j_1}^4 \ldots h_{j_m}^4 \gamma_{j_1, \ldots, j_m} (\mathbf{x}_h; h_{j_1}, \ldots, h_{j_m}).
    \end{align}
Here, in particular, denoting 
\(I_h^{(i_1,\ldots,i_m)}
:= \{ x \in I^d : x_{i_k} = j h_{i_k},\ 0 \le j \le h_{i_k}^{-1} \}\),
\[
\beta_j(\cdot; h_1,\ldots,h_{j-1},h_{j+1},\ldots,h_d)
\in L^\infty\bigl(I_h^{(1,\ldots,j-1,j+1,\ldots,d)}\bigr),
\quad
\gamma_{j_1,\ldots,j_m}(\cdot; h_{j_1},\ldots,h_{j_m})
\in L^\infty\bigl(I_h^{(j_1,\ldots,j_m)}\bigr)
\]
are uniformly bounded independent of $\mathbf{h}$.


\end{thm}

\begin{proof}
    The proof is given in Section \ref{sec:Chapter 5}.

\end{proof}

\begin{rem}
There are two alternatives of interpolating the finite difference solutions on the tensor meshes to
construct the sparse combination solution.
In the sequence "extrapolate -- interpolate -- combine", we require a fourth order interpolant with appropriate
error splitting to maintain the overall fourth order accuracy from the preceding extrapolation step.
It is shown in \cite{hendricks2017error} that iterative univariate cubic spline interpolation satisfies such an error  splitting
for sufficiently smooth functions.

In the sequence "interpolate -- extrapolate -- combine", the extrapolation step may lift the second order of multi-linear interpolation to fourth order, and give a fourth order sparse grid solution after combination, assuming that
an expansion as on the right-hand side of \eqref{expansion_intro} holds for the interpolation error.
We are positive that this can be proven by combining the techniques to derive the interpolation error in
\cite{reisinger2013analysis} with the techniques for the more refined expansion for the finite difference error in the proof of Theorem \ref{mainresult}. However, the lengthy analysis would significantly extend the length of this paper without any likely new insights and is hence omitted.
The numerical tests in Section \ref{sec:appl} confirm that the extrapolation and combination of a second order finite difference solution with multi-linear interpolation does indeed give fourth order convergence.
\end{rem}

In summary, by extrapolating a second‐order finite‐difference scheme to fourth‐order and then applying the sparse‐grid combination formula, we obtain a global error bound  
\[
|u - \hat u_n| = O\bigl((n+2d-1)^{d-1}2^{-4n}\bigr),
\]
while using only \(O(n^{d-1}2^n)\) degrees of freedom instead of \(O(2^{nd})\),
making the higher‐order combination solution both accurate and dimensionally efficient.






\section{Multivariate Extrapolation and Higher Order Combination}
\label{sec:Chapter 4}

The goal of this section is to construct a higher‐order sparse‐grid combination formula by first developing a fourth‐order extrapolation on Cartesian subgrids. We first apply extrapolation on the classical Cartesian grid to build higher order solutions, and only then combine them via the sparse grid formula. By a preliminary extrapolation step on each sparse subgrid with mesh sizes $\mathbf{h}$ one ensures that all mixed error terms up to the desired order vanish. 

For the extrapolation step, we can take a more abstract view.
Let $\mathbf{u}_h$ be a quantity that depends on a vector of parameters \( (h_{1}, \ldots, h_{d}) \)
and $u$ its limit as $h_k \rightarrow 0$ for all $k=1,\ldots, d$.

In the context of this paper, $u$ is the solution to a linear PDE 
on the unit cube $[0,1]^d$, and $\mathbf{u}_h$ the finite difference approximation of order 2 accuracy on a tensor product grid with mesh sizes $\mathbf{h} = (h_1, \ldots, h_d)$, with $u_{\mathbf{h}} : [0,1]^d \to \mathbb{R}$ be its multilinear interpolant, with $u_{\mathbf{h}}(\mathbf{j}\odot \mathbf{h}) = (\mathbf{u}_h)_{\mathbf{j}}$. 

In the setting of Theorem \ref{thm:error_expansion} and assuming additionally that $w$ can be expanded further,
\begin{align*}
    u - u_{\mathbf{h}} 
    & = \sum_{m=1}^d \sum_{\substack{\{j_1, \ldots, j_m\}\\ \subset \{1, \ldots, d\}} } h_{j_1}^2 \ldots h_{j_m}^2 \beta_{j_1, \ldots, j_m} + R(h),
\end{align*}
where the remainder \(R(h)\) collects the higher–order and mixed terms. 
By applying this multivariate extrapolation on each Cartesian subgrid we aim at exactly eliminating all mixed‐quadratic terms \(h_{j_1}^2 \ldots h_{j_m}^2\), so that the remaining leading error on each grid is of order \(O(h^4)\).

Recall the notation from \eqref{eq:extra} and immediately thereafter.
Then $u_{\mathbf{h}}^{(i_1, \ldots, i_k)}$ satisfies the error expansion
    \begin{align*}
    u - u_{\mathbf{h}}^{(i_1, \ldots, i_k)} = \sum_{m=1}^d \sum_{\substack{\{j_1, \ldots, j_m\} \subset \{1, \ldots, d\} } } 4^{-N(\mathbf{i}, \mathbf{j})}h_{j_1}^2 \ldots h_{j_m}^2 \beta_{j_1, \ldots, j_m} + \bar{R}(h),
    \end{align*}
    where $N(\mathbf{i}, \mathbf{j}) = \| \{i_1, \ldots, i_k\} \cap \{j_1, \ldots, j_m\} \|$, and \(\bar R(h)\) denotes the remainder term collecting all higher‐order error contributions beyond the explicit mixed‐quadratic terms.

Define 
    \begin{align*}
    N_l \coloneqq \sum_{\substack{\{i_1, \ldots, i_k\} \subset \{1, \ldots, d\} \\ | \{i_1, \ldots, i_k\} \cap \{j_1, \ldots, j_m\}| = l }} \!\!\!\!\!\!\!\!  1 \quad = \quad \binom{m}{l} \binom{d-m}{k-l}
    \end{align*}
    that counts the number of subgrids for which $N(\mathbf{i}, \mathbf{j}) = l$. 
Then 
    \begin{align*}
        \sum_{\substack{\{i_1, \ldots, i_k\} \\ \subset \{1, \ldots, d\}}} \sum_{\substack{\{j_1, \ldots, j_m\} \\ \subset \{1, \ldots, d\} }} 4^{-N(\mathbf{i}, \mathbf{j})} \beta_{j_1, \ldots, j_m} h_{j_1}^2 \ldots h_{j_m}^2 = \sum_{l=0}^d N_l 4^{-l} \sum_{\substack{ \{j_1, \ldots, j_m\} \\ \subset \{1, \ldots, d\} }} \beta_{j_1, \ldots, j_m} h_{j_1}^2 \ldots, h_{j_m}^2
    \end{align*}
    so that 

    \begin{align*}
        \sum_{\substack{ \{i_1, \ldots, i_k\} \\ \subset \{1, \ldots, d\} }} u - u_{\mathbf{h}}^{(i_1, \ldots, i_k)} = \sum_{m=1}^d \left[ \sum_{l= \max\{0, m+k-d\} }^{\min\{m,k\}} 4^{-l} N_l \right] \sum_{\substack{ \{j_1, \ldots, j_m\} \\ \subset \{1, \ldots, d\} }} \beta_{j_1, \ldots, j_m} h_{j_1}^2 \ldots h_{j_m}^2 + \bar{R}(h).
    \end{align*}

The \(m\)th order term $h_{j_1}^2 \ldots h_{j_m}^2$ will vanish only if (by swapping order of summation)  
    \begin{align*}
        \sum_{k=0}^d\alpha_k \sum_{l = \max\{ 0, m +´k-d\} }^{ \min \{m,k\} } 4^{-l} N_l = \sum_{k=0}^d \alpha_k \sum_{l = \max\{ 0, m +´k-d\} }^{ \min \{m,k\} } 4^{-l} \binom{m}{l} \binom{d-m}{k-l} = 0.
    \end{align*}
    Solving this system of linear equations for $\alpha_k$,
    $$
    \begin{cases}
        \sum_{k=0}^d \alpha_k \binom{d}{k} = 1, \\
        \sum_{k=0}^d \alpha_k \sum_{l = \max\{ 0, m +´k-d\} }^{ \min \{m,k\} } 4^{-l} \binom{m}{l} \binom{d-m}{k-l} = 0,
    \end{cases}
    $$
    we conjecture (from explicit calculation for small $d$) the closed form solution $\alpha_k = \frac{(-4)^k}{(-3)^d}$.

In the subsequent proof of Theorem \ref{HOexp}, we apply those weights in the extrapolated solution $\tilde{u}_h$ and show that all second-order mixed contributions indeed vanish, leaving only pure quartic higher-order terms in the error expansion. This constitutes the key result of the section, since it upgrades a nominally second-order discretization into an effectively fourth-order sparse-grid approximation.  
We first show the key cancellation step in a lemma.

\begin{lem}
\label{lem:cancel}
Let $\alpha_k := (-4)^k/(-3)^d$ for $0 \le k \le d$. Then, for arbitrary $\beta:\mathbb{R}^{d-1}\to\mathbb{R}$ and indices $i_j\in\{0,1\}$ for $1\le j\le d$, we have
\[
\sum_{k=0}^d \alpha_k
\!\!\!\sum_{\substack{i_1+\cdots+i_d=k\\ i_j\in\{0,1\}}}
\! 4^{-i_1}\, \beta\!\bigl(i_2,\ldots,i_d \bigr) \;=\; 0.
\]
\end{lem}

\begin{proof}
Substituting $h_\ell = 2^{-i_\ell}$ for $\ell=2,\dots,d$ gives
\begin{align}
\sum_{k=0}^d \alpha_k \!\!\!\sum_{i_1+\cdots+i_d=k}\!\! 4^{-i_1} \,\beta(h_2,\ldots,h_d)
&= \frac{1}{(-3)^d}\sum_{k=0}^d (-4)^k \!\!\!\sum_{i_1+\cdots+i_d=k}\!\! 4^{-i_1}\,\beta(i_2,\ldots,i_d) \nonumber \\
&= \frac{1}{(-3)^d}\sum_{k=0}^d (-1)^k \!\!\!\sum_{i_1+\cdots+i_d=k}\!\! 4^{i_2+\ldots+i_d}\,\beta(i_2,\ldots,i_d) \nonumber \\
&= \frac{1}{(-3)^d}\sum_{k=0}^d (-1)^k \!\!\!\sum_{i_1+\cdots+i_d=k}\!\!
\tilde{\beta}(i_2,\ldots,i_d), \label{telescope}
\end{align}
where we have defined $\tilde{\beta}(i_2,\ldots,i_d) = 4^{i_2+\ldots+i_d}\, \beta(i_2,\ldots,i_d)$.

Since the last sum in \eqref{telescope} does not depend on $i_1$, for all $0<k<d$ we can split it as
\[
\sum_{i_1+\cdots+i_d=k}\tilde{\beta}(i_2,\ldots,i_d)
= \sum_{i_2+\cdots+i_d=k-1}\tilde{\beta}(i_2,\ldots,i_d)
+ \sum_{i_2+\cdots+i_d=k}\tilde{\beta}(i_2,\ldots,i_d),
\]
and similarly at the beginning and end of the range.
Hence \eqref{telescope} is a telescoping sum, which proves the claim.
\end{proof}

\begin{proof}

[of Theorem \ref{HOexp}]

First, a zero order consistency condition holds:
\[
\sum_{k=0}^d \alpha_k \binom{d}{k}
= \frac{1}{(-3)^d}\sum_{k=0}^d (-4)^k \binom{d}{k}
= \frac{(1-4)^d}{(-3)^d} = 1.
\]
The remainder follows from Lemma~\ref{lem:cancel} and
\begin{align*}
u - \tilde u_h
&= \sum_{m=1}^d \;\sum_{\substack{\{j_1,\ldots,j_m\} \\ \subset\{1,\ldots,d\}}}
\sum_{k=0}^d \alpha_k
\!\!\!\sum_{\substack{\sum_\ell i_\ell = k\\ i_\ell\in\{0,1\}}}
\gamma_{j_1,\ldots,j_m}\!\bigl(h_{j_1+i_{j_1}},\ldots,h_{j_m+i_{j_m}}\bigr)
\prod_{r=1}^m h_{j_r+i_{j_r}}^{4} \\
&= \sum_{m=1}^d \;\sum_{\substack{\{j_1,\ldots,j_m\} \\ \subset\{1,\ldots,d\}}}
\biggl(\prod_{r=1}^m h_{j_r}^{4}\biggr)
\underbrace{\sum_{k=0}^d \frac{(-4)^k}{(-3)^d}
\!\!\!\sum_{\substack{\sum_\ell i_\ell = k\\ i_\ell\in\{0,1\}}}
\gamma_{j_1,\ldots,j_m}\!\bigl(h_{j_1}2^{-i_{j_1}},\ldots,h_{j_m}2^{-i_{j_m}}\bigr)
\, 16^{-(i_{j_1}+\cdots+i_{j_m})}}_{=:~\tilde\gamma_{j_1,\ldots,j_m}(h_{j_1},\ldots,h_{j_m})}.
\end{align*}
Estimating the coefficient gives
\[
\bigl|\tilde\gamma_{j_1,\ldots,j_m}\bigr|
\;\le\; \sum_{k=0}^d \frac{K\, 4^k}{3^d}\binom{d}{k}
\;=\; K\,\frac{(1+4)^d}{3^d}
\;=\; K\,\frac{5^d}{3^d}.
\qedhere
\]

\end{proof}

\section{Error Expansion}\label{sec:Chapter 5}

In this section, we prove our main result, Theorem~\ref{mainresult}.
A formal proof (i.e.\ without detailed verification of the required regularity and other properties) for the case $d=2$ is given in Appendix \ref{proof-2d}. It highlights the general principle of the proof but avoids some of the more technical steps of the general case.

We present the arguments in an operator–theoretic framework so that it applies beyond any specific model, and we specialize afterwards to finite–difference discretizations of linear boundary value problems (in particular, the Poisson problem; see Section \ref{sec:appl}).

Let $A: C^{\infty}([0,1]^d, \mathbb{R}) \rightarrow C^{\infty}([0,1]^d, \mathbb{R})$ be some linear operator. Fix a vector of mesh sizes $\mathbf{h}$, let $I_{h_j}:=\{0,h_j,2h_j,\ldots,1\}$ and $I_h:=\prod_{j=1}^d I_{h_j}$, and denote by $\mathcal X_h$ the space of grid functions on $I_h$. Let $A_h:\mathcal X_h\rightarrow \mathcal X_h$ be a second–order discretization of $A$.
For some regular enough function $f:[0,1]^d \rightarrow \mathbb{R}$, let  $u\in C^\infty(I^d)$ solve $Au=f$, and let $u_h\in\mathcal X_h$ solve $A_h u_h=f_h$, where $f_h$ is the standard restriction of $f$ to $I_h$.

We interpret $u_h$ as the grid approximation of $u$ on $I_h$. Then the goal of this section is to establish error expansions of the form 
\begin{align}\label{exp}
    u - u_h = \sum_{j=1}^d h_j^2 \beta_j(h_1, \ldots, h_{j-1}, h_{j+1}, \ldots, h_d) + \sum_{m=1}^d \sum_{\substack{ \{ j_1, \ldots, j_m \} \\ \subset \{ 1, \ldots, d \} }} h_{j_1}^4 \ldots h_{j_m}^4 \gamma_{j_1, \ldots, j_m} (h_{j_1}, \ldots, h_{j_m}). 
\end{align} 

A crucial first step for understanding the error $u - u_h$ in finite difference schemes is typically to consider the truncation error $A_h [u-u_h] = A_h u - f_h$. Consequently, expanding this residual in powers of the mesh sizes is a key ingredient of our error analysis.
Then the global discretisation error can be viewed as the discrete operator's inverse acting on this local error. One might then naively invert $A_h$ to obtain the desired error expansion for $u - u_h$. 
However, since $A_h^{-1}$ depends non-trivially on the mesh-sizes $(h_1, \ldots, h_d)$, this approach cannot directly yield an expansion of the desired form \eqref{exp}.

To overcome this obstable, we introduce a theoretical error correction scheme. We will prove by induction
the following claim: 
    \begin{claim}\label{claim}
        For every $1 \leq m \leq d$ and $0 \leq n \leq m$, there exist functions $\beta_{j}^{(n,m)}(\cdot; h_{1}, \ldots, h_{j-1}, h_{j+1}, \ldots , h_{d})$, $\beta_{j, j_2, \ldots, j_k}^{(n,m)}(\cdot; h_{j_2}, \ldots, h_{j_k}), \gamma_{j_1, \ldots, j_k}^{(n,m)} (\cdot; h_{j_1}, \ldots, h_{j_k})$ and ${\sigma_{i_1, \ldots, i_l}^{k_1, \ldots, k_{m-l}}}^{(n,m)}(\cdot; h_{i_1}, \ldots, h_{i_l}) $ such that
        \begin{align}\label{equality?}
            & A_h \left[ u - \sum_{k=1}^{m-1} \sum_{\substack{ \{ j_1, \ldots, j_k \} \\ \subset \{ 1, \ldots, d \} }} h_{j_1}^4 \ldots h_{j_k}^4 \gamma^{(n,m)}_{j_1, \ldots, j_k} (h_{j_1}, \ldots, h_{j_k}) -  \sum_{j = 1}^d h_{j}^2 \beta_{j}^{(n,m)} (h_{1}, \ldots, h_{j-1}, h_{j+1}, \ldots, h_{d}) \right] - f_h \notag \\
            &= \sum_{k=m+1}^{d} \sum_{\substack{ \{ j_1, \ldots, j_k \} \\ \subset \{ 1, \ldots, d \} }} h_{j_1}^4 \ldots h_{j_k}^4 \gamma^{(n,m)}_{j_1, \ldots, j_k} ( h_{j_1}, \ldots, h_{j_k}) + \sum_{k=m+1}^{d} \sum_{\substack{\{j\} \sqcup \{ j_2, \ldots, j_k\} \\ \subset \{1, \ldots, d\} } } h_j^2 h_{j_2}^2 \ldots h_{j_k}^2 \beta_{j, j_2, \ldots, j_k}^{(n,m)}( h_{j_2}, \ldots, h_{j_k}) \notag \\
            & \quad + \sum_{l=n}^m \sum_{\substack{ \{i_1, \ldots, i_l\} \sqcup \{k_1, \ldots, k_{m-l}\} \\ \subset \{1, \ldots, d\} } }h_{k_1}^2 \dots h_{k_{m-l}}^2 h_{i_1}^4 \dots h_{i_l}^4 {\sigma_{i_1, \ldots, i_l}^{k_1, \ldots, k_{m-l}}}^{(n,m)}(h_{i_1}, \dots, h_{i_l}).
        \end{align}
Moreover, all functions 
\(
\beta_{j}^{(n,m)},\quad 
\beta_{j,j_2,\ldots,j_k}^{(n,m)},\quad
\gamma_{j_1,\ldots,j_k}^{(n,m)},\quad
\sigma_{i_1,\ldots,i_l}^{k_1,\ldots,k_{m-l},(n,m)}
\)
have the following properties:

\begin{itemize}
  \item They are $C^\infty$ functions in the continuous variables of the partial grids, and all their derivatives vanish on the boundary.
  \item They are even and $C^\infty$ in each mesh–size variable on which they depend.
\end{itemize}
    \end{claim}

Before proving the claim, we show how the expansion (\ref{exp}) and hence Theorem \ref{mainresult} directly follows from the claim. 

\begin{proof}[Proof of Theorem \ref{mainresult}]
    Taking $n = m=d$ in \eqref{equality?}, 
    \begin{align*}
            & A_h \left[ u - \sum_{k=1}^{d-1} \sum_{\substack{ \{ j_1, \ldots, j_k \} \\ \subset \{ 1, \ldots, d \} }} h_{j_1}^4 \ldots h_{j_k}^4 \gamma^{(d,d)}_{j_1, \ldots, j_k} ( h_{j_1}, \ldots, h_{j_k}) -  \sum_{j = 1}^d h_{j}^2 \beta_{j}^{(d,d)} ( h_{1}, \ldots, h_{j-1}, h_{j+1}, \ldots, h_{d}) \right] - f_h \\
            &= h_{1}^4 \ldots h_{d}^4 \gamma^{(d,d)}_{1, \ldots, d} ( h_{1}, \ldots, h_{d}) \\
            \implies & u - u_h = \sum_{k=1}^{d} \sum_{\substack{ \{ j_1, \ldots, j_k \} \\ \subset \{ 1, \ldots, d \} }} h_{j_1}^4 \ldots h_{j_k}^4 \bar{\gamma}^{(d,d)}_{j_1, \ldots, j_k} ( h_{j_1}, \ldots, h_{j_k}) + \sum_{j=1}^{d} h_{j}^2 \beta_{j}^{(d,d)} (h_{1}, \ldots, h_{j-1}, h_{j+1}, \ldots, h_{d}),
        \end{align*}
    where we multiply by \(A_h^{-1}\) and define 
\(
\bar{\gamma}^{(d,d)}_{1,\ldots,d}(\cdot) := A_h^{-1}\gamma^{(d,d)}_{1,\ldots,d}(\cdot),
\)
using the fact that, by Assumption~\ref{newasumption}, the operator \(A_h\) has an inverse that is bounded independently of $(h_1, \ldots, h_d)$.

    This proves (\ref{HOexpansion}) in Theorem \ref{mainresult}.
    \end{proof}

It remains to show Claim \ref{claim} by induction in $1 \leq m \leq d$ and $0 \leq n \leq m$.
The overall strategy of the proof, used both in the base case and the induction step, and motivated by the two-dimensional argument of Section~4, can be summarized as follows.

\begin{enumerate}
    \item Identify the terms on the right-hand side of (\ref{equality?}) that do not yet conform to the desired expansion structure.
    \item Construct an appropriate semi-discrete problem whose truncation error replicates the identified term.
    \item Subtract the solution of the semi–discrete problem from the left–hand side of (\ref{equality?}), thereby cancelling the identified term at the cost of introducing, on the right–hand side, the truncation error of the semi–discrete problem.
    \item Show that the truncation error has the required form, i.e., it consists only of terms compatible with the target expansion by repeatedly applying Taylor expansions and using the fact that the truncation errors are smooth, even functions of the mesh sizes.
\end{enumerate}
This systematic approach is applied throughout the proof to iteratively reduce the right-hand side to the desired form, thereby establishing the inductive claim.

We introduce the following family of semi-discrete operators that split the discretization by dimension. 

\begin{defn}
    For any subset $\{i_1, \ldots, i_m\} \subset \{1, \ldots, d\}$ define the semi discrete operator $A_h^{(i_1, \ldots, i_m)}$, that discretizes the linear operator $A$ in the directions $i_1, \ldots, i_m$ and leaves the remaining directions continuous. 
    These give the systems of equations 
    \[
    A_h^{(i_1, \ldots, i_m)} u^{(i_1, \ldots, i_m)} = f_h^{(i_1, \ldots, i_m)} 
    \]
    with boundary conditions, on the partial grid $I_h^{(i_1, \ldots, i_m)} \coloneqq \left\{ x \in I^d : x_{i_k} = j h_{i_k}, 0 \leq j \leq h_{i_k}^{-1}, 1 \leq k \leq m\right\}$. 
\end{defn}

We make the following assumptions. 
For $m=0$, we define for compactness $A_h^{(i_1, \ldots, i_m)} = A$ and $\{i_1, \ldots, i_m\} = \emptyset$.

\begin{assump}\label{newasumption}

For any right-hand side \(g(\cdot ;h_{i_{1}},\ldots,h_{i_{m}}):
I_{h}^{(i_{1},\ldots,i_{m})} \mapsto \mathbb{R}\) that is smooth in the spatial variables with all derivatives vanishing on the boundary, and smooth and even with respect to the mesh sizes, we assume that the 
semi–discrete problem
\[
A_{h}^{(i_{1},\ldots,i_{m})} w(\cdot;h_{i_{1}},\ldots,h_{i_{m}}) 
    = g(\cdot;h_{i_{1}},\ldots,h_{i_{m}})
\]
admits a solution \(w\) on 
\(I_{h}^{(i_{1},\ldots,i_{m})}\) with the following properties:
\begin{enumerate}[label=(\arabic*)]
    \item \(w\) is smooth in all spatial variables.
    \item \(w\) satisfies 
homogeneous boundary conditions on
    \(\partial I_{h}^{(i_{1},\ldots,i_{m})}
    = I_{h}^{(i_{1},\ldots,i_{m})} \cap \partial I^{d}\):
    \[
        \partial^{\alpha} w = 0
        \qquad \text{on }\partial I_{h}^{(i_{1},\ldots,i_{m})}
        \quad\text{for all multi-indices }\alpha.
    \]
    That is, the solution and all of its mixed derivatives vanish whenever any continuous 
coordinate takes the value \(0\) or \(1\). 
    \item \(w\) is smooth and even in each meshsize component 
    \(h_{i_k}\). 
    
\end{enumerate}
\end{assump}

Then, given any regular enough function $w$, considering the error $[A_h^{(i_1, \ldots, i_m)} - A_h] w$ will allow us to isolate components of the discretisation error of operator $A_h$ only in the non-discretized directions. 
Hence finding appropriate expansions of such errors in terms of powers of the mesh sizes will be a main focus in our analysis.

\begin{assump}\label{assumption}
We assume that any function 
$w(\cdot;h_{i_1},\ldots,h_{i_m})$ satisfying the smoothness and 
boundary conditions stated in Assumption~\ref{newasumption} 
admits an error representation of the form
\[
\left[ \mathbf{A}_h^{(i_1, \ldots, i_m)} - \mathbf{A}_h \right] w(\cdot; h_{i_1}, \ldots, h_{i_m})
= \sum_{k=1}^{d-m} \sum_{\substack{ \{j_1, \ldots, j_k\} \\ \subset \{1, \ldots, d \} - \{ i_1, \ldots, i_m \} }} 
h_{j_1}^2 \cdots h_{j_k}^2 \,
\hat{\sigma}_{j_1, \ldots, j_k}^{i_1, \ldots, i_m}(\cdot ; h_{j_1}, \ldots, h_{j_k}, h_{i_1}, \ldots, h_{i_m}),
\]
where each coefficient
\[
\hat{\sigma}_{j_1,\dots,j_k}^{i_1, \ldots, i_m}:I^{d-m}\times[0,1]^{k+m}\;\longrightarrow\;\mathbb{R}
\]
satisfies the following properties:

\begin{enumerate}[label=(\arabic*)]
  \item For every fixed mesh‐size tuple $(h_{j_1},\ldots,h_{j_k}, h_{i_1}, \ldots, h_{i_m})$, 
  \\the function
  \(
    y\;\mapsto\;\hat{\sigma}_{j_1,\dots,j_m}^{i_1, \ldots, i_m}\bigl(y;\,h_{j_1},\dots,h_{j_k}, h_{i_1}, \dots, h_{i_m}\bigr)
  \)
  is smooth in $I^{d-m}$.

    \item For every fixed mesh–size tuple,  
  $\hat{\sigma}_{j_1,\dots,j_k}^{i_1,\ldots,i_m}$ vanishes on the boundary 
  $\partial I^{d-m}$, and all of its mixed derivatives in the continuous variables
  vanish on $\partial I^{d-m}$.

  \item  For every fixed $y\in I^{d-m}$, the map
  \(
    (h_{j_1},\dots,h_{j_k}, h_{i_1}, \dots, h_{i_m})
    \;\mapsto\;
    \hat\sigma_{j_1,\dots,j_m}^{i_1, \ldots, i_m}\bigl(y;\,h_{j_1},\dots,h_{j_k}, h_{i_1}, \dots, h_{i_m}\bigr)
  \)
  is a smooth even function on $[0,1]^{k+m}$. 
\end{enumerate}
\end{assump}

Throughout the remainder of this section, all statements and proofs 
are carried out under Assumptions~\ref{newasumption} and~\ref{assumption}. 

In Section \ref{sec:appl}, we verify these assumptions in detail for the Poisson equation under appropriate regularity hypotheses. For completeness, we also include below a brief discussion explaining why such 
assumptions are reasonable in more general settings.
\\For arbitrary boundary data, the zero trace hypothesis in 
Assumption~\ref{newasumption} causes no loss of generality.  
Although sparse grid analysis is usually carried out under zero trace boundary conditions, this is merely a convenient normalisation. In practice, high-dimensional PDEs are posed on $\mathbb{R}^d$ and their solutions decay smoothly at infinity. After localisation to a bounded box, the 
imposed boundary values differ from a smooth zero trace function only by a small perturbation, which has a controlled effect on the sparse grid error. Thus, the zero trace assumption is not restrictive. For a detailed discussion, see \cite{reisinger2013analysis}. 
\\For typical symmetric (central) discretizations, the regularity and symmetry properties in Assumption \ref{assumption} are natural. They can be justified by carrying out a sufficiently high order Taylor expansion of the exact solution in both the spatial variables and the mesh sizes, together with the use of second order consistent finite difference stencils. A rigorous verification follows from tracking higher order remainder terms and exploiting the symmetry of the difference quotients.

One can show that any such even smooth function $\sigma_{j_1, \ldots, j_m}$ admits a Taylor expansion in the squares of its arguments:

\begin{lem}\label{evenexp}
Let $\sigma:[-1,1]^d \rightarrow \mathbb{R}$ be a smooth even function of $h_1, \ldots, h_d$, that is $\sigma(h_1, \ldots, h_d) = \sigma\left((-1)^{i_1}h_1, \ldots, (-1)^{i_d} h_d\right)$, for all $(i_1, \ldots, i_d) \in \{0,1\}^d$. 

Then there exist smooth even functions $\delta_{j_1, \ldots, j_m}:[-1,1]^m \rightarrow \mathbb{R}$ such that 

\begin{align*}
\sigma( h_{1}, \ldots, h_{d}) = \sum_{m=0}^d \sum_{\substack{ \{j_1, \ldots, j_m\} \\ \subset \{1, \ldots, d\} }} h_{j_1}^2 \ldots h_{j_m}^2 \delta_{j_1, \ldots, j_m} (h_{j_1}, \ldots, h_{j_m}),
\end{align*}
where by abuse of notation, for $m=0$, $\delta_{j_1, \ldots, j_m} \eqqcolon \delta_0$  denotes a constant. 
\end{lem}

\begin{proof}
We perform induction in $d$, using at each step a standard even Taylor expansion in the last variable. In the base case $d=1$, one writes
\[
\sigma(h)=\tfrac12\bigl[\sigma(h)+\sigma(-h)\bigr]
=\sigma(0)+h^2\,\delta(h),
\]
where $\delta$ is a smooth even function.
In the inductive step, a symmetrized Taylor expansion in \(h_d\) and the induction hypothesis in dimension \(d-1\) yield the full decomposition at once. 
\end{proof}

The assumptions stated above, combined with the expansion result for smooth even functions given in the previous lemma, allow us to establish the following technical result. This lemma will play a key role in the proof of the main theorem, where it will be applied repeatedly to handle the structure of various truncation error terms. For clarity and ease of reference, we state it here separately, and provide a detailed proof in the Appendix \ref{Appendix}.

\begin{lem}\label{technicallemma}
Let $\{ i_1, \ldots, i_n \} \subseteq \{k_1, \ldots, k_p\} \subseteq \{1, \ldots, d\}$ and $\{j_1, \ldots, j_m\} \cap \{k_1, \ldots, k_p\} = \emptyset$. 

Let 
\(
\phi(\cdot; h_{i_1}, \ldots, h_{i_n})
\)
be any sufficiently regular function, smooth in its spatial variables and 
smooth and even in each mesh-size component.  
Suppose that the semi-discrete problem
\[
A_h^{(k_1, \ldots, k_p)} \psi(\cdot; h_{k_1}, \ldots, h_{k_p})
    = \phi(\cdot; h_{k_1}, \ldots, h_{k_p})
\]
admits a solution $\psi$ satisfying the regularity requirements of 
Assumptions ~\ref{newasumption} and ~\ref{assumption}.  
Then the corresponding solution 
\(
\psi(\cdot; h_{k_1}, \ldots, h_{k_p}) \in C^\infty(I^d)
\)
inherits the same smoothness and evenness properties in the mesh sizes, and 
the truncation error satisfies:
\begin{align}\label{RHS}
    & h_{j_1}^2 \cdots h_{j_m}^2 \, h_{i_1}^4 \cdots h_{i_n}^4 \, \big[ A_h^{(k_1, \ldots, k_p)} - A_h \big] \psi(h_{k_1}, \ldots, h_{k_p})  \\
    & = \sum_{q = m+n+1}^d \sum_{ \substack{ \{l_1, \ldots, l_q\} \\ \subset \{1, \ldots, d\} }} 
    h_{l_1}^4 \cdots h_{l_q}^4 \, \gamma_{l_1, \ldots, l_q}(\cdot; h_{l_1}, \ldots, h_{l_q}) \notag \\
    & \quad + \sum_{q = m+n+1}^d \sum_{ \substack{ \{l\} \sqcup \{l_2, \ldots, l_q\} \\ \subset \{1, \ldots, d\} }} 
    h_l^2 h_{l_2}^2 \cdots h_{l_q}^2 \, \beta_{l l_2 \ldots l_q}(h_{l_2}, \ldots, h_{l_q}) \notag \\
    & \quad + \sum_{q = n+1}^{m+n} \sum_{ \substack{ \{s_1, \ldots, s_q\} \sqcup \{r_1, \ldots, r_{m+n-q}\} \\ \subset \{1, \ldots, d\} }} 
    h_{r_1}^2 \cdots h_{r_{m+n-q}}^2 \, h_{s_1}^4 \cdots h_{s_q}^4 \, 
    \sigma_{s_1, \ldots, s_q}^{r_1, \ldots, r_{m+n-q}}(h_{s_1}, \ldots, h_{s_q}).\notag
\end{align}
Here, $\gamma_{l_1, \ldots, l_q}$, $\beta_{l l_2 \ldots l_q}$, and $\sigma_{s_1, \ldots, s_q}^{r_1, \ldots, r_{m+n-q}}$ are smooth functions that are also smooth and even in the mesh-size variables.
\end{lem}
\begin{proof}
The proof proceeds by expanding the truncation error using Assumption~\ref{assumption} and Lemma~\ref{evenexp} and organising all resulting terms into one of the three structures appearing on the right-hand side of (\ref{RHS}).
A detailed justification of this decomposition is provided in Appendix~\ref{Appendix}.
\end{proof}

We now state and prove a second technical lemma needed for the main proof.

\begin{lem}\label{lem:base}
For smooth $u$, there exist smooth functions $\sigma_j^{(0,1)}$, $\gamma_{j_1, \ldots, j_k}^{(0,1)}$, $\beta_{j,j_2, \ldots , j_k}^{(0,1)}$ and constants ${\sigma^j}^{(0,1)}$ such that
    \begin{align*}
     & A_h u - f_h = \sum_{j=1}^d h_{j}^2 {\sigma^j}^{(0,1)} + \sum_{j=1}^d h_j^4 \sigma_j^{(0,1)}(h_j)\\
     & \qquad +   \sum_{k=2}^d \sum_{\substack{ \{j_1, \dots, j_k\} \\ \subset \{1, \dots, d\} }} h_{j_1}^4 \dots h_{j_k}^4 \gamma_{j_1, \dots, j_k}^{(0,1)}(h_{j_1}, \dots, h_{j_k}) + \sum_{k=2}^d \sum_{\substack{ \{j\}\sqcup \{j_2, \dots, j_k\} }} \beta_{j,j_2,\dots,j_k}^{(0,1)}(h_{j_2}, \dots, h_{j_k}). 
    \end{align*}
\end{lem}
\begin{proof}
    Consider the truncation error of the second order finite difference approximation as per Assumption \ref{assumption},
    \[
    A_h u - f_h = \sum_{m=1}^d \sum_{\substack{ \{j_1, \ldots, j_m\} \\ \subset \{1, \ldots, d\} }} h_{j_1}^2 \ldots h_{j_m}^2 \sigma_{j_1, \ldots, j_m}(h_{j_1}, \ldots, h_{j_m}).
    \]

Moreover, since each coefficient \(\sigma_{j_1,\dots,j_m}(h_{j_1},\dots,h_{j_m})\) is a smooth, even function of the mesh sizes, we may apply a Taylor expansion in the variables \(h_j\). For \(m \geq 2\), each product
\begin{align*}
        & h_{j_1}^2 \ldots h_{j_m}^2 \sigma_{j_1, \ldots, j_m}(h_{j_1}, \ldots, h_{j_m}) \\
        = &  h_{j_1}^2 \ldots h_{j_m}^2 \left[ \sigma_0^{^{(j_1, \dots, j_m)}} + \sum_{k=1}^{m-1} \sum_{\substack{\{i_1, \ldots, i_k\} \\ \subset \{j_1, \ldots, j_m\} }} h_{i_1}^2 \ldots h_{i_k}^2 \sigma_{i_1, \ldots, i_k}^{(j_1, \dots, j_m)}(h_{i_1}, \ldots, h_{i_k}) + h_{j_1}^2 \ldots h_{j_m}^2 \sigma_{j_1, \dots, j_m}^{(j_1, \ldots, j_m)}(h_{j_1}, \ldots, h_{j_m}) \right] \\
        = & h_{j_1}^2 h_{j_2}^2 \ldots h_{j_m}^2 \sigma_0^{(j_1,\ldots, j_m)} + \sum_{k=1}^{m-1} \sum_{\substack{ \{i_1, \ldots, i_k\} \\ \subsetneq \{j_1, \ldots, j_m \} }} h_{j_1}^2 \ldots h_{j_m}^2 \hat{\sigma}_{i_1, \ldots, i_k}^{(j_1, \ldots, j_m)}(h_{i_1}, \ldots, h_{i_k}) + h_{j_1}^4 \ldots h_{j_m}^4 \sigma_{j_1, \ldots, j_m}^{(j_1, \ldots, j_m)}(h_{j_1}, \ldots, h_{j_m}) 
    \end{align*}
    with $\hat{\sigma}_{i_1, \ldots, i_k}^{(j_1, \ldots, j_m)}(h_{i_1, \ldots, i_k}) \coloneqq h_{i_1}^2 \ldots h_{i_k}^2 \sigma_{i_1, \ldots, i_k}^{(j_1, \dots, j_m)}(h_{i_1}, \ldots, h_{i_k})$, splits into a pure fourth‐order term for the form 
\[
h_{j_1}^4\cdots h_{j_m}^4\;\gamma_{j_1,\dots,j_m}(h_{j_1},\dots,h_{j_m})
\]
and mixed second‐order terms of the form
\[
h_j^2\,h_{j_2}^2\cdots h_{j_m}^2\;\beta_{j,j_2,\dots,j_m}(h_{j_2},\dots,h_{j_m}).
\]
For $m=1$, 
\[
h_j^2 \sigma_j(h_j) = h_j^2 \sigma_0 +  h_j^4 \sigma_j(h_j).
\]
Summing over all index‐sets then expresses 
\(\,A_h u - f_h\)\ 
precisely as a linear combination of fourth‐order and mixed second‐order contributions as in the statement.
\end{proof}


We are now in a position to prove Claim \ref{claim} and therefore complete the proof of Theorem \ref{mainresult}.

    \begin{proof}[Proof of Claim \ref{claim}]
    We organise the proof as a two–parameter induction.  
For each \(0 \le n \le m\) and \( 1 \le m \le d\), let \(P(n,m)\) denote the desired statement.

We begin by verifying the base case \(P(0,1)\).  
For fixed \(m\), we then perform induction in \(n\), showing that 
\(P(n,m) \Rightarrow P(n+1,m)\) for all \(0 \le n < m\); thus, once \(P(0,m)\) is known, this yields \(P(1,m),\ldots,P(m,m)\).  
Finally, we advance from level \(m\) to \(m+1\) by proving that \(P(m,m) \Rightarrow P(0,m+1)\), which initiates the next sequence.  
These steps together yield \(P(n,m)\) for all \(0 \le n \le m \) and \( 1 \le m \le d\), completing the proof.

    The base case $P(0,1)$ is given by Lemma \ref{lem:base}.

    For the induction in \(n\), assume \(P(n,m)\) for some \(n<m\), i.e.\ that equality~\eqref{equality?} holds. To show $P(n+1,m)$ holds, we need to cancel the terms 
    \[
    \sum_{\substack{ \{i_1, \dots, i_n\} \sqcup \{k_1, \dots, k_{m-n}\} \\ \subset \{1, \dots, d\} }} h_{k_1}^2 \dots h_{k_{m-n}}^2 h_{i_1}^4 \dots h_{i_n}^4 {\sigma_{i_1, \dots, i_n}^{k_1, \dots, k_{m-n}}}^{(m,n)}(h_{i_1}, \dots, h_{i_n}). 
    \]
    For that purpose, introduce the following sequence of semi-discrete problems: 
    \[
    A_h^{(i_1, \dots, i_n)} \hat{\sigma}_{i_1,\dots,i_n}^{k_1, \dots, k_{m-n}}(h_{i_1}, \dots, h_{i_n}) = {\sigma_{i_1,\dots,i_n}^{k_1, \dots, k_{m-n}}}^{(n,m)}(h_{i_1}, \dots, h_{i_n}). 
    \]
    By Lemma \ref{technicallemma}, the truncation errors of such a sequence of semi-discrete problems satisfy: 
    \begin{align}\label{thisthis}
       & h_{k_1}^2 \dots h_{k_{m-n}}^2 h_{i_1}^4 \dots h_{i_n}^4 \left[ A_h^{(i_1,\dots,i_n)} - A_h\right] \hat{\sigma}_{i_1,\dots,i_n}^{k_1,\dots,k_{m-n}}(h_{i_1},\dots,h_{i_n}) \notag \\
       & = \sum_{k = m+1}^d \sum_{\substack{ \{l_1,\dots,l_k\} \\ \subset \{1,\dots, d\} }} h_{l_1}^4 \dots h_{l_k}^4 \gamma_{l_1,\dots, l_k}^{k_1,\dots,k_{m-n},i_1,\dots,i_n}(h_{l_1},\dots,h_{l_k}) \notag \\
       & \qquad + \sum_{k = m+1}^d \sum_{ \substack{ \{l\}\sqcup\{l_2,\dots,l_k\} \\ \subset \{1,\dots, d\} } } h_l^2 h_{l_2}^2 \dots h_{l_k}^2 \beta_{l,l_2,\dots,l_k}^{k_1,\dots,k_{m-n},i_1,\dots,i_n}(h_{l_2},\dots,h_{l_k}) \notag \\
       & \qquad + \sum_{k = n+1}^{m} \sum_{ \substack{ \{s_1,\dots,s_k\} \sqcup \{r_1,\dots,r_{m-k}\} \\ \subset \{1,\dots, d\} } } h_{r_1}^2 \dots h_{r_{m-k}}^2 h_{s_1}^4 \dots h_{s_k}^4 \sigma_{s_1,\dots,s_k}^{r_1,\dots,r_{m-k},k_1,\dots,k_{m-n},i_1,\dots,i_n}(h_{s_1},\dots,h_{s_k}). 
    \end{align}
    
    The induction step in $n$ now follows at once by subtracting such semi-discrete problems, applying the induction hypothesis to go from the second to the third line and using (\ref{thisthis}) to go from the fourth to the fifth equality: 
    \begin{align*}
        & A_h \left[ u - \sum_{k=1}^{m-1} \sum_{ \substack{ \{j_1, \ldots, j_k\} \\ \subset \{1, \ldots, d\} } } h_{j_1}^4 \ldots h_{j_k}^4 \gamma_{j_1, \ldots, j_k}^{(n+1,m)}(h_{j_1}, \ldots, h_{j_k}) - \sum_{j=1}^d h_j^2 \beta_j^{(n+1,m)}(h_1, \ldots, h_{j-1}, h_{j+1}, \ldots, h_d)\right] - f_h \\
        \coloneqq & A_h \left[ u - \sum_{k=1}^{m-1} \sum_{ \substack{ \{j_1, \ldots, j_k\} \\ \subset \{1, \ldots, d\} } } h_{j_1}^4 \ldots h_{j_k}^4 \gamma_{j_1, \ldots, j_k}^{(n,m)}(h_{j_1}, \ldots, h_{j_k}) - \sum_{j=1}^d h_j^2 \beta_j^{(n,m)}(h_1, \ldots, h_{j-1}, h_{j+1}, \ldots, h_d)\right] - f_h\\
        & \qquad - A_h \left[ \sum_{ \substack{ \{i_1,\dots,i_n\}\sqcup \{k_1,\dots,k_{m-n}\} \\ \subset \{1,\dots,d\} }} h_{k_1}^2 \dots h_{k_{m-n}}^2 h_{i_1}^4 \dots h_{i_n}^4 \hat{\sigma}_{i_1,\dots,i_n}^{k_1,\dots,k_{m-n}}(h_{i_1},\dots,h_{i_n})\right] \\ 
        = & \sum_{k=m+1}^{d} \sum_{\substack{ \{ j_1, \ldots, j_k \} \\ \subset \{ 1, \ldots, d \} }} h_{j_1}^4 \ldots h_{j_k}^4 \gamma^{(n,m)}_{j_1, \ldots, j_k} ( h_{j_1}, \ldots, h_{j_k}) + \sum_{k=m+1}^{d} \sum_{\substack{\{j\} \sqcup \{ j_2, \ldots, j_k\} \\ \subset \{1, \ldots, d\} } } h_j^2 h_{j_2}^2 \ldots h_{j_k}^2 \beta_{j, j_2, \ldots, j_k}^{(n,m)}( h_{j_2}, \ldots, h_{j_k}) \notag \\
            & \quad + \sum_{l=n+1}^m \sum_{\substack{ \{s_1, \ldots, s_l\} \sqcup \{r_1, \ldots, r_{m-l}\} \\ \subset \{1, \ldots, d\} } }h_{r_1}^2 \dots h_{r_{m-l}}^2 h_{s_1}^4 \dots h_{s_l}^4 {\sigma_{s_1, \ldots, s_l}^{r_1, \ldots, r_{m-l}}}^{(n,m)}(h_{s_1}, \dots, h_{s_l}) \\
            & \qquad + \sum_{ \substack{ \{i_1,\dots,i_n\} \sqcup \{k_1,\dots,k_{m-n}\} \\ \subset \{1,\dots d\} } } h_{k_1}^2 \dots h_{k_{m-n}}^2 h_{i_1}^4 \dots h_{i_n}^4 \left[ {\sigma_{i_1,\dots,i_n}^{k_1,\dots,k_{m-n}}}^{(n,m)} - A_h \hat{\sigma}_{i_1,\dots,i_n}^{k_1,\dots,k_{m-n}} \right](h_{i_1},\dots,h_{i_n})\\
        = & \sum_{k=m+1}^{d} \sum_{\substack{ \{ j_1, \ldots, j_k \} \\ \subset \{ 1, \ldots, d \} }} h_{j_1}^4 \ldots h_{j_k}^4 \gamma^{(n,m)}_{j_1, \ldots, j_k} ( h_{j_1}, \ldots, h_{j_k}) + \sum_{k=m+1}^{d} \sum_{\substack{\{j\} \sqcup \{ j_2, \ldots, j_k\} \\ \subset \{1, \ldots, d\} } } h_j^2 h_{j_2}^2 \ldots h_{j_k}^2 \beta_{j, j_2, \ldots, j_k}^{(n,m)}( h_{j_2}, \ldots, h_{j_k}) \notag \\
            & \quad + \sum_{l=n+1}^m \sum_{\substack{ \{s_1, \ldots, s_l\} \sqcup \{r_1, \ldots, r_{m-l}\} \\ \subset \{1, \ldots, d\} } }h_{r_1}^2 \dots h_{r_{m-l}}^2 h_{s_1}^4 \dots h_{s_l}^4 {\sigma_{s_1, \ldots, s_l}^{r_1, \ldots, r_{m-l}}}^{(n,m)}(h_{s_1}, \dots, h_{s_l}) \\
            & \qquad + \sum_{ \substack{ \{i_1,\dots,i_n\} \sqcup \{k_1,\dots,k_{m-n}\} \\ \subset \{1,\dots d\} } } h_{k_1}^2 \dots h_{k_{m-n}}^2 h_{i_1}^4 \dots h_{i_n}^4 \left[ A_h^{(i_1,\dots,i_n)}- A_h  \right]\hat{\sigma}_{i_1,\dots,i_n}^{k_1,\dots,k_{m-n}}(h_{i_1},\dots,h_{i_n})\\
        = & \sum_{k=m+1}^{d} \sum_{\substack{ \{ j_1, \ldots, j_k \} \\ \subset \{ 1, \ldots, d \} }} h_{j_1}^4 \ldots h_{j_k}^4 \gamma^{(n+1,m)}_{j_1, \ldots, j_k} ( h_{j_1}, \ldots, h_{j_k}) + \sum_{k=m+1}^{d} \sum_{\substack{\{j\} \sqcup \{ j_2, \ldots, j_k\} \\ \subset \{1, \ldots, d\} } } h_j^2 h_{j_2}^2 \ldots h_{j_k}^2 \beta_{j, j_2, \ldots, j_k}^{(n+1,m)}( h_{j_2}, \ldots, h_{j_k}) \notag \\
            & \quad + \sum_{l=n+1}^m \sum_{\substack{ \{s_1, \ldots, s_l\} \sqcup \{r_1, \ldots, r_{m-l}\} \\ \subset \{1, \ldots, d\} } }h_{r_1}^2 \dots h_{r_{m-l}}^2 h_{s_1}^4 \dots h_{s_l}^4 {\sigma_{s_1, \ldots, s_l}^{r_1, \ldots, r_{m-l}}}^{(n+1,m)}(h_{s_1}, \dots, h_{s_l}), \\
    \end{align*}
    where
\begin{align*}
\gamma^{(n+1,m)}_{j_1,\ldots,j_k}
  &\coloneqq \gamma^{(n,m)}_{j_1,\ldots,j_k}
     + \sum_{\substack{\{i_1,\ldots,i_n\}\,\sqcup\,\{k_1,\ldots,k_{m-n}\}}}
       \gamma^{k_1,\ldots,k_{m-n},i_1,\ldots,i_n}_{j_1,\ldots,j_k}, \\[0.4em]
\beta^{(n+1,m)}_{j,j_2,\ldots,j_k}
  &\coloneqq \beta^{(n,m)}_{j,j_2,\ldots,j_k}
     + \sum_{\substack{\{i_1,\ldots,i_n\}\,\sqcup\,\{k_1,\ldots,k_{m-n}\}}}
       \beta^{k_1,\ldots,k_{m-n},i_1,\ldots,i_n}_{j,j_2,\ldots,j_k}, \\[0.4em]
{\sigma_{s_1, \ldots, s_l}^{r_1, \ldots, r_{m-l}}}^{(n+1,m)}
  &\coloneqq {\sigma_{s_1, \ldots, s_l}^{r_1, \ldots, r_{m-l}}}^{(n,m)}
     + \sum_{\substack{\{i_1,\ldots,i_n\}\,\sqcup\,\{k_1,\ldots,k_{m-n}\}}}
       \sigma^{r_1,\ldots,r_{m-l},k_1,\ldots,k_{m-n},i_1,\ldots,i_n}_{s_1,\ldots,s_l}.
\end{align*}
This shows that $P(n,m) \implies P(n+1,m)$, $n < m$, completing the induction in $n$. To finish the proof, we only need to establish the remaining link $P(m,m) \implies P(0,m+1)$. 

Assume $P(m,m)$ holds. To show $P(0,m+1)$ holds, we first need to cancel the terms 
\[
\sum_{ \substack{ \{i_1,\dots,i_m\} \\ \subset \{1, \dots, d\} } } h_{i_1}^4 \dots h_{i_m}^4 \sigma_{i_1, \dots i_m}^{(m,m)}(h_{i_1}, \dots, i_m).
\]
For that purpose, we introduce the following sequence of semi-discrete problems: 
\[
A_h^{(i_1,\dots,i_m)} \gamma_{i_1,\dots,i_m}^{(0,m+1)}(h_{i_1},\dots,h_{i_m}) = \sigma_{i_1,\dots,i_m}^{(m,m)}(h_{i_1},\dots,h_{i_m}). 
\]
By Lemma \ref{technicallemma}, the truncation errors of such a sequence of semi-discrete problems satisfy:  
\begin{align}\label{label}
    & h_{i_1}^4 \dots h_{i_m}^4 \left[ A_h^{(i_1,\dots,i_m)} - A_h \right] \gamma_{i_1,\dots,i_m}^{(0,m+1)} (h_{i_1},\dots,h_{i_m}) \notag\\
    = & \sum_{q = m+1}^d \sum_{ \substack{ \{l_1,\dots,l_q\} \\ \subset \{1,\dots,d\} } } h_{l_1}^4 \dots h_{l_q}^4 \gamma_{l_1,\dots,l_q}(h_{l_1}, \dots, h_{l_q}) \notag\\
    & \qquad + \sum_{q = m+1}^d \sum_{ \substack{ \{l\}\sqcup \{l_2,\dots,l_q\} \\ \subset \{1,\dots, d\} } } h_l^2 h_{l_2}^2 \dots h_{l_q}^2 \beta_{l,l_2,\dots,l_q}(h_{l_2}, \dots, h_{l_q}). 
\end{align}

Now by subtracting such a sequence of semi-discrete problems from the RHS of $P(m,m)$, applying the induction hypothesis to go from the second to the third equality and using (\ref{label}) to go from the fourth to the fifth line: 
\begin{align*}
    & A_h \left[ u - \sum_{k=1}^{m} \sum_{\substack{ \{ j_1, \ldots, j_k \} \\ \subset \{ 1, \ldots, d \} }} h_{j_1}^4 \ldots h_{j_k}^4 \gamma^{(0,m+1)}_{j_1, \ldots, j_k} (h_{j_1}, \ldots, h_{j_k}) -  \sum_{j = 1}^d h_{j}^2 \beta_{j}^{(0,m+1)} (h_{1}, \ldots, h_{j-1}, h_{j+1}, \ldots, h_{d}) \right] - f_h \notag \\
    \coloneqq & A_h \left[ u - \sum_{k=1}^{m-1} \sum_{\substack{ \{ j_1, \ldots, j_k \} \\ \subset \{ 1, \ldots, d \} }} h_{j_1}^4 \ldots h_{j_k}^4 \gamma^{(m,m)}_{j_1, \ldots, j_k} (h_{j_1}, \ldots, h_{j_k}) -  \sum_{j = 1}^d h_{j}^2 \beta_{j}^{(m,m)} (h_{1}, \ldots, h_{j-1}, h_{j+1}, \ldots, h_{d}) \right] - f_h \notag \\
    & \qquad - A_h \left[ \sum_{ \substack{\{i_1,\dots,i_m\} \\ \subset \{1,\dots,d\} }} h_{j_1}^4 \dots h_{j_m}^4 \gamma_{j_1,\dots,j_m}^{(0,m+1)}(h_{j_1}, \dots, h_{j_m}) \right] \\
    &= \sum_{k=m+1}^{d} \sum_{\substack{ \{ j_1, \ldots, j_k \} \\ \subset \{ 1, \ldots, d \} }} h_{j_1}^4 \ldots h_{j_k}^4 \gamma^{(m,m)}_{j_1, \ldots, j_k} ( h_{j_1}, \ldots, h_{j_k}) + \sum_{k=m+1}^{d} \sum_{\substack{\{j\} \sqcup \{ j_2, \ldots, j_k\} \\ \subset \{1, \ldots, d\} } } h_j^2 h_{j_2}^2 \ldots h_{j_k}^2 \beta_{j, j_2, \ldots, j_k}^{(m,m)}( h_{j_2}, \ldots, h_{j_k}) \notag \\
            & \quad + \sum_{\substack{ \{i_1, \ldots, i_m\} \\ \subset \{1, \ldots, d\} } } h_{i_1}^4 \dots h_{i_m}^4 \left[ \sigma_{i_1, \ldots, i_m}^{(m,m)} - A_h \gamma_{i_1,\dots,i_m}^{(0,m+1)}\right](h_{i_1}, \dots, h_{i_m})\notag \\
    &= \sum_{k=m+1}^{d} \sum_{\substack{ \{ j_1, \ldots, j_k \} \\ \subset \{ 1, \ldots, d \} }} h_{j_1}^4 \ldots h_{j_k}^4 \gamma^{(m,m)}_{j_1, \ldots, j_k} ( h_{j_1}, \ldots, h_{j_k}) + \sum_{k=m+1}^{d} \sum_{\substack{\{j\} \sqcup \{ j_2, \ldots, j_k\} \\ \subset \{1, \ldots, d\} } } h_j^2 h_{j_2}^2 \ldots h_{j_k}^2 \beta_{j, j_2, \ldots, j_k}^{(m,m)}( h_{j_2}, \ldots, h_{j_k}) \notag \\
            & \quad + \sum_{\substack{ \{i_1, \ldots, i_m\} \\ \subset \{1, \ldots, d\} } } h_{i_1}^4 \dots h_{i_m}^4 \left[ A_h^{(i_1,\dots,i_m)} - A_h \right] \gamma_{i_1,\dots,i_m}^{(0,m+1)}(h_{i_1}, \dots, h_{i_m})\notag \\
    &= \sum_{k=m+1}^{d} \sum_{\substack{ \{ j_1, \ldots, j_k \} \\ \subset \{ 1, \ldots, d \} }} h_{j_1}^4 \ldots h_{j_k}^4 \left[\gamma^{(m,m)}_{j_1, \ldots, j_k} + \gamma_{j_1,\dots,j_k}\right] ( h_{j_1}, \ldots, h_{j_k}) \\
    & \qquad + \sum_{k=m+1}^{d} \sum_{\substack{\{j\} \sqcup \{ j_2, \ldots, j_k\} \\ \subset \{1, \ldots, d\} } } h_j^2 h_{j_2}^2 \ldots h_{j_k}^2 \left[ \beta_{j, j_2, \ldots, j_k}^{(m,m)} + \beta_{j,j_2,\dots,j_k}\right]( h_{j_2}, \ldots, h_{j_k}) \notag \\
     &\eqqcolon \sum_{k=m+1}^{d} \sum_{\substack{ \{ j_1, \ldots, j_k \} \\ \subset \{ 1, \ldots, d \} }} h_{j_1}^4 \ldots h_{j_k}^4 \hat{\gamma}_{j_1, \ldots, j_k} ( h_{j_1}, \ldots, h_{j_k}) + \sum_{k=m+1}^{d} \sum_{\substack{\{j\} \sqcup \{ j_2, \ldots, j_k\} \\ \subset \{1, \ldots, d\} } } h_j^2 h_{j_2}^2 \ldots h_{j_k}^2 \hat{\beta}_{j, j_2, \ldots, j_k}( h_{j_2}, \ldots, h_{j_k}).  \notag \\
\end{align*}

Note that we can rewrite the above equality in the following convenient way: 

\begin{align}\label{ref}
    & A_h \left[ u - \sum_{k=1}^{m} \sum_{\substack{ \{ j_1, \ldots, j_k \} \\ \subset \{ 1, \ldots, d \} }} h_{j_1}^4 \ldots h_{j_k}^4 \gamma^{(0,m+1)}_{j_1, \ldots, j_k} (h_{j_1}, \ldots, h_{j_k}) -  \sum_{j = 1}^d h_{j}^2 \beta_{j}^{(0,m+1)} (h_{1}, \ldots, h_{j-1}, h_{j+1}, \ldots, h_{d}) \right] - f_h \notag \\
            &= \sum_{k=m+2}^{d} \sum_{\substack{ \{ j_1, \ldots, j_k \} \\ \subset \{ 1, \ldots, d \} }} h_{j_1}^4 \ldots h_{j_k}^4 \hat{\gamma}_{j_1, \ldots, j_k} ( h_{j_1}, \ldots, h_{j_k}) + \sum_{k=m+2}^{d} \sum_{\substack{\{j\} \sqcup \{ j_2, \ldots, j_k\} \\ \subset \{1, \ldots, d\} } } h_j^2 h_{j_2}^2 \ldots h_{j_k}^2 \hat{\beta}_{j, j_2, \ldots, j_k}( h_{j_2}, \ldots, h_{j_k}) \notag \\
            & \quad + \sum_{\substack{\{j_1,\dots,j_{m+1}\} \\ \subset \{1,\dots, d\}}} h_{j_1}^4 \dots h_{j_{m+1}}^4 \hat{\gamma}_{j_1,\dots,j_{m+1}}(h_{j_1},\dots,h_{j_{m+1}}) \notag\\
            & \quad + \sum_{ \substack{ \{j\}\sqcup \{j_2,\dots,j_{m+1}\} \\ \subset \{1,\dots, d\} } } h_j^2 h_{j_2}^2 \dots h_{j_{m+1}}^2 \hat{\beta}_{j,j_2,\dots,j_{m+1}}(h_{j_2}, \dots, h_{j_{m+1}}). 
\end{align}
By assumption, $\hat{\beta}_{j,j_2,\dots,j_{m+1}}$ are smooth even functions of the mesh sizes, so by Lemma \ref{evenexp}: 
\begin{align*}
    & h_j^2 h_{j_2}^2 \dots h_{j_{m+1}}^2 \hat{\beta}_{j,j_2,\dots,j_{m+1}}(h_{j_2},\dots,h_{j_{m+1}})\\
    =& h_j^2 h_{j_2}^2 \dots h_{j_{m+1}}^2 \left[ \beta_0^{j,j_2,\dots,j_{m+1}} + \sum_{l=1}^m \sum_{\substack{ \{r_1,\dots,r_l\} \\ \subset \{j_2,\dots,j_{m+1}\} } } h_{r_1}^2 \dots h_{r_l}^2 \beta_{r_1,\dots,r_l}^{j,j_2,\dots,j_{m+1}}(h_{r_1}, \dots, h_{r_l}) \right]\\
    =& \sum_{l=0}^m \sum_{ \substack{ \{r_1,\dots,r_l\} \sqcup\{j\}\sqcup \{s_1, \dots, s_{m-l}\} \\ = \{j\}\sqcup\{j_2,\dots,j_{m+1}\} }} h_j^2 h_{s_1}^2 \dots h_{s_{m-l}}^2 h_{r_1}^4 \dots h_{r_l}^4 \beta_{r_1,\dots,r_l}^{j,j_2,\dots,j_{m+1}}(h_{r_1}, \dots, h_{r_l}). 
\end{align*}
Substituting the above into (\ref{ref}) finally yields $P(0,m+1)$: 
\begin{align}
& A_h \left[ u - \sum_{k=1}^{m} \sum_{\substack{ \{ j_1, \ldots, j_k \} \\ \subset \{ 1, \ldots, d \} }} h_{j_1}^4 \ldots h_{j_k}^4 \gamma^{(0,m+1)}_{j_1, \ldots, j_k} (h_{j_1}, \ldots, h_{j_k}) -  \sum_{j = 1}^d h_{j}^2 \beta_{j}^{(0,m+1)} (h_{1}, \ldots, h_{j-1}, h_{j+1}, \ldots, h_{d}) \right] - f_h \notag \\
            &= \sum_{k=m+2}^{d} \sum_{\substack{ \{ j_1, \ldots, j_k \} \\ \subset \{ 1, \ldots, d \} }} h_{j_1}^4 \ldots h_{j_k}^4 \gamma^{(0,m+1)}_{j_1, \ldots, j_k} ( h_{j_1}, \ldots, h_{j_k}) + \sum_{k=m+2}^{d} \sum_{\substack{\{j\} \sqcup \{ j_2, \ldots, j_k\} \\ \subset \{1, \ldots, d\} } } h_j^2 h_{j_2}^2 \ldots h_{j_k}^2 \beta^{(0,m+1)}_{j, j_2, \ldots, j_k}( h_{j_2}, \ldots, h_{j_k}) \notag \\
 & \qquad + \sum_{l=0}^{m+1} \sum_{\substack{ \{r_1,\dots,r_l\} \sqcup \{s_1,\dots,s_{m+1-l}\} \\ \subset \{1,\dots, d\} }} h_{s_1}^2 \dots h_{s_{m+1-l}}^2 h_{r_1}^4 \dots h_{r_{l}}^4 {\sigma_{r_1,\dots,r_l}^{s_1,\dots,s_{m+1-l}}}^{(0,m+1)}(h_{r_1}, \dots, h_{r_l}) \notag, 
\end{align}
where 
\begin{align*}
    \gamma_{j_1,\dots,j_k}^{(0,m+1)} & \coloneqq \hat{\gamma}_{j_1,\dots,j_k} \text{ for $k \geq m+2$,}\\
    \beta_{j,j_2,\dots,j_k}^{(0,m+1)} & \coloneqq \hat{\beta}_{j,j_2,\dots,j_k} \text{ for $k \geq m+2$,}\\
    \sigma_{r_1,\dots,r_{m+1}}^{(0,m+1)} &\coloneqq \hat{\gamma}_{r_1,\dots, r_{m+1}}, \\
    {\sigma_{r_1,\dots,r_l}^{s_1, \dots, s_{m+1-l}}}^{(0,m+1)} & \coloneq \beta_{r_1,\dots,r_l}^{s_1,\dots,s_{m+1-l},r_1,\dots,r_l}. 
\end{align*}

    This completes the induction argument and hence the proof of Claim \ref{claim}.  
\end{proof}

The principal remaining step to justify the above formal derivation for a specific PDE is to verify the validity of Assumption~\ref{assumption}, which guarantees the required smooth, even expansions of the truncation errors. 

Although these regularity hypotheses may seem restrictive, they are satisfied by a wide class of linear elliptic and parabolic PDEs. 
In the next section, we verify that Assumotions ~\ref{newasumption} and ~\ref{assumption} hold for the Poisson equation under certain regularity hypotheses, and other important model problems likewise fit into this framework.

\section{Application to the Poisson Problem}
\label{sec:appl}

In this section, we first confirm that the Poisson equation falls within the theoretical framework required for the higher‐order combination technique for sufficiently regular data, and then present numerical experiments illustrating its performance.

\subsection{Theoretical Error Analysis for the Poisson Equation}

We consider the Poisson problem 
\begin{align}
\sum_{k=1}^d \partial_k^2 u(x) &= f(x), \quad x \in I^d, \notag \\
u(x) &= 0, \qquad x \in \partial I^d.
\label{Poisson}
\end{align}
\\Equivalently, let \(A:C^\infty(I^d)\to C^\infty(I^d)\) be the linear operator
\((Aw)(x):=\sum_{k=1}^d \partial_k^2 w(x)\).
Then \eqref{Poisson} can be written compactly as
$Au=f$ in $I^d$, $u|_{\partial I^d}=0$.
Assumption \ref{newasumption} requires that the continuous
solution possesses uniformly bounded mixed derivatives of sufficiently high order and that these derivatives vanish on the boundary. As is well known,
solutions of elliptic problems on the cube do not automatically inherit such smoothness from the right-hand side $f$ in the presence of boundary corners. Therefore,
compatibility conditions on 
$f$ must be imposed.

Let $f_{\mathbf{k}}$ denote the sine–series coefficients of $f$, where
$\mathbf{k}=(k_1,\ldots,k_d)\in\mathbb{N}^d$. A regularity result, presented explicitly for the case \(n=2\) in \cite{reisinger2013analysis}, states that if
\begin{equation}\label{eqq}
    \sum_{\mathbf{k}\in\mathbb{N}^d}
\;\sum_{\|\mathbf{i}\|_1\le \lfloor d/2\rfloor}
\frac{k_1^{4n+i_1}\cdots k_d^{4n+i_d}}{(k_1^2+\cdots+k_d^2)^2}\,
f_{\mathbf{k}}^{\,2} < \infty,
\end{equation}
then the solution of \eqref{Poisson} satisfies $u\in X^d_{2n}$, i.e.\ all mixed
derivatives of total order $\le 2n$ are continuous and vanish on $\partial I^d$.
Thus, the Poisson solution has the smoothness required for sparse–grid analysis
provided the sine coefficients of $f$ decay sufficiently quickly.

These regularity requirements may seem restrictive, but they are in fact common in practical high-dimensional PDE applications. Typically, the underlying
problem is posed on $\mathbb{R}^d$ and only localised to a bounded box for
computability. The unbounded solution is smooth, and the localised problem with
asymptotic boundary conditions differs from a smooth function only by a small
perturbation. Under standard stability assumptions, this perturbation produces
a controlled and negligible contribution to the sparse grid error. A more
detailed discussion of this localisation argument is given in \cite{reisinger2013analysis}. 

Although we have invoked $C^\infty$ uniformly bounded mixed derivatives of all orders for notational convenience, in the presence of the continuous (and discrete) maximum principle, one can actually reduce the requirement to bounded mixed derivatives up to sufficiently high order. Moreover, for the Poisson problem, it is enough that all mixed derivatives containing at least one even–order derivative in some coordinate, up to sufficiently high order, remain bounded and vanish on the boundary. As the analysis in this section
shows, these are precisely the higher-order mixed derivatives that appear in the truncation error formula of the finite-difference approximation of the Laplacian. This weaker assumption still guarantees both stability and the
validity of the truncation error expansion. Our use of the stronger smoothness framework simply avoids repeatedly specifying the precise highest derivative
needed.

We discretize \(A\) via the standard second order central finite difference operator. By a simple Truncated Taylor expansion, for any smooth function $w$ and any $x \in I^d$,
\begin{equation}\label{evenremark}
\partial_{k}^2 w(x) = \frac{1}{h_k^2} \left[ w(x + h_k \mathbf{e}_k) - 2 w(x) + w(x - h_k \mathbf{e}_k) \right] - \frac{h_k^2}{12} \partial_{k}^4 w(x) - 2 \frac{h_k^4}{6!} \left[ \partial_{k}^6 w(x_*^{(k)}(h_k)) \right],
\end{equation}
where $x_{*}^{(k)}(h_k)$ is on the line from $x-h_k \mathbf{e}_k$ to from $x+h_k \mathbf{e}_k$,
with $\mathbf{e}_k$ the $k$-th unit vector. 

We discretise the domain on a Cartesian grid with mesh sizes \(h_1,\ldots,h_d\).
For each coordinate \(j\), set \(I_{h_j}:=\{0,h_j,2h_j,\ldots,1\}\).
Occasionally, for a subset \(\{i_1, \ldots, i_m\}\subseteq\{1,\ldots,d\}\), we shall discretise only in the directions in \( \{i_1, \ldots, i_m\} \) and keep the complementary coordinates continuous. 
The partial grid is
\[
I_h^{(i_1, \ldots, i_m)}:=\Bigl\{x\in I^d:\ x_j\in I_{h_j} \text{for }j\in \{i_1, \ldots, i_m\},\ \text{and }x_j\in[0,1]\ \text{for }j\notin \{i_1, \ldots, i_m\}\Bigr\}.
\]
and we write 
\[
w_{\mathbf{i}}(y) := w(x_1, \ldots, x_d)
\quad \text{where} \quad
x_j = i_j h_j \text{ for } j \in \{i_1, \ldots, i_m\}, \quad 
y = (x_{j_1}, \ldots, x_{j_{d-m}}) \in [0,1]^{d-m},
\]
\[
\text{and} \quad \{j_1, \ldots, j_{d-m}\} \sqcup \{i_1, \ldots, i_m\} = \{1, \ldots, d\}.
\]
Then, by \eqref{evenremark}, the standard second order central difference semi-discretization of \(A\) in the directions \(\{i_1,\ldots,i_m\}\) evaluated at the semi-discrete node \(x\) is
\[
\bigl(A_h^{(i_1,\ldots,i_m)} w\bigr)(x)
=\sum_{k\in \{i_1, \ldots, i_m\}}\frac{w_{\mathbf i+\mathbf e_k}(y)-2w_{\mathbf i}(y)+w_{\mathbf i-\mathbf e_k}(y)}{h_k^2}
\;+\;\sum_{k\notin \{i_1, \ldots, i_m\}}\partial_k^2 w(x).
\]

The regularity result above (\ref{eqq}) ensures that the Poisson problem admits a solution \(u\) with exactly the
smoothness and boundary behaviour required in Assumption~\ref{newasumption}.  

For a fixed subset of discrete directions \(\{i_1,\ldots,i_m\}\), the semi–discrete problem
\[
A_h^{(i_1,\ldots,i_m)} w(\cdot; h_{i_1},\ldots,h_{i_m}) \;=\; g(\cdot; h_{i_1},\ldots,h_{i_m})
\]
is an elliptic equation in the continuous coordinates, coupled with a finite difference
Laplacian in the remaining ones.  
Therefore, if the semi-discrete right-hand side \(g\) satisfies the same type of regularity
assumptions as \(f\) (namely smoothness and vanishing derivatives of all orders on \(\partial I^d\)),
then the solution \(w\) of the semi–discrete problem inherits the same properties.

Moreover, the symmetry of the finite difference stencil implies that the dependence of \(w\)
on each mesh size \(h_{i_k}\) is smooth and even, as required by Assumption~\ref{newasumption}. 

Thus every semi–discrete Poisson subproblem satisfies Assumption~\ref{newasumption} provided
its right–hand side \(g\) satisfies the corresponding smoothness and boundary–vanishing
conditions.

Having established that the semi–discrete Poisson subproblems admit smooth
solutions satisfying Assumption~\ref{newasumption}, we now
verify that their truncation error also possesses the structure specified in
Assumption~\ref{assumption}.

By the Taylor expansion in even powers \eqref{evenremark}, the local discretisation
error of the semi-discrete operator $A_h^{(i_1,\ldots,i_m)}$ satisfies
\begin{align*}
   \left[ A_h^{(i_1, \ldots, i_m)} - A_h \right] w(x) & =  \sum_{k \notin \{i_1, \ldots, i_m\}} \left( \partial_k^2w(x) - \frac{w_{\mathbf i+\mathbf e_k}(y)-2w_{\mathbf i}(y)+w_{\mathbf i-\mathbf e_k}(y)}{h_k^2} \right) \\
  &= \sum_{ k \notin \{i_1, \ldots, i_m\}} h_k^2 \frac{- \partial_k^4 w(x)}{12} + \sum_{ k \notin \{i_1, \ldots, i_m\} } h_k^4 \frac{- \partial_k^6 w(x_*^{(k)}(h_k))}{6!}\\
  & \eqqcolon \sum_{k \notin \{i_1, \ldots, i_m\}} h_k^2 \beta_k(x) + \sum_{k \notin \{i_1, \ldots, i_m\}} h_k^4 \gamma_k(x; h_k).
\end{align*}
Here, $\beta_k(x) = -\tfrac{1}{12}\partial_k^4 w(x)$ is independent of all mesh
sizes, while $\gamma_k(x;h_k)$ is a smooth even function of $h_k$, as all odd powers cancel in \eqref{evenremark}.  

Because $w$ and all its mixed derivatives vanish on $\partial I^{d-m}$ by
Assumption~5.3, and since $\beta_k$ and $\gamma_k(\cdot;h_k)$ are constructed
solely from derivatives of $w$, it follows that $\beta_k$ and
$\gamma_k(\cdot;h_k)$ also vanish on $\partial I^{d-m}$ together with all their
mixed derivatives.  
Thus the truncation error decomposes additively into coordinate-wise
$h_k^2$ and $h_k^4$ contributions, with coefficient functions that are smooth in
the continuous variables, even and smooth in the mesh sizes, and satisfy the
required homogeneous boundary conditions. This verifies Assumption~5.4 for the
semi–discrete Poisson operator.

The following theorem is therefore an immediate consequence of the abstract error expansion results established earlier.
\begin{thm}\label{thm:poisson}
    Let $u$ be a smooth solution of the Poisson problem (\ref{Poisson}) and $\mathbf{u}_h$ the finite difference approximation on a grid $\mathbf{x}_h$. Then 
    \[
    u(\mathbf{x}_h) - \mathbf{u}_h = \sum_{j=1}^d h_j^2 \beta_j(\mathbf{x}_h; h_1, \ldots, h_{j-1}, h_{j+1}, \ldots, h_d) + \sum_{m=1}^d \sum_{\substack{ \{j_1, \ldots, j_m\} \\ \subset \{1, \ldots, d\} }} h_{j_1}^4 \ldots h_{j_m}^4 \gamma_{j_1, \ldots, j_m}(\mathbf{x}_h; h_{j_1}, \ldots, h_{j_m})
    \]
    where $\beta_j$ and $\gamma_{j_1, \ldots, j_m}$ are uniformly bounded functions. 
\end{thm}

\begin{proof}
Since we have shown in this section that the Poisson problem on $I^d$ meets the hypotheses of Theorem~\ref{mainresult}, the error expansion follows immediately.
\end{proof}

\subsection{Numerical Results}

We test the higher-order sparse combination technique on the Poisson problem with smooth data, specifically with
\[
f(x) = d\,\pi^2 \prod_{i=1}^d \sin(\pi x_i),
\]
so that
\[
u(x) = -\prod_{i=1}^d \sin(\pi x_i)
\]
is smooth on $[0,1]^d$.  
This choice of $f$ also ensures that the weaker versions of Assumptions~\ref{newasumption} and~\ref{assumption} hold: in this setting it is sufficient that all mixed derivatives containing at least one even-order derivative in some coordinate vanish on the boundary. These conditions are slightly weaker than the full assumptions but remain fully compatible with the analysis above.

{
\color{red}
}

We evaluate the hierarchical surplus
\(
\hat\epsilon_n = \lvert \hat u_{n+1}(x^*) - \hat u_n(x^*) \rvert
\)
at the interior point \(x^* = (0.25, 0.5, \dots, 0.25, 0.5)\) and refine the discretization up to level \(n = 12\) in order to assess convergence. This allows us both to verify that the proposed higher-order sparse grid (HO-SG) method attains the expected fourth-order accuracy and to benchmark its decay against the standard sparse grid (SG), the standard full grid (FG), and the higher-order full grid (HO-FG) constructed from FG by standard Richardson extrapolation. Figures~\ref{dims_surplus_level} and~\ref{dims_surplus_cost} report the surplus decay with respect to the refinement level and the computational cost, respectively, for dimensions \(d = 1, \dots, 7\).

From the Surplus vs Level plot (Figure~\ref{dims_surplus_level}), we observe that the standard FG and SG methods exhibit decay rates consistent with second-order convergence, while both higher-order methods display significantly steeper slopes. In particular, the HO-SG surplus aligns closely with the reference slope \(-4\) over a wide range of levels and dimensions, confirming that the extrapolation-based combination technique successfully lifts the underlying second-order discretization to fourth-order accuracy. The dependence on the dimension is weak: only a very mild deterioration of the observed convergence rate is observed. This can be explained by the smoothness of the Poisson test problem, which satisfies the regularity assumptions underlying sparse grid error estimates. Consequently, the HO-SG method consistently outperforms the standard SG and retains a clear higher-order character in the asymptotic regime.

The Surplus vs Cost plot (Figure~\ref{dims_surplus_cost}) further demonstrates the practical advantage of HO-SG. While full-grid methods (FG and HO-FG) become rapidly prohibitive as the dimension grows, the sparse grid constructions achieve comparable accuracy at significantly reduced cost. Crucially, HO-SG combines this cost efficiency with a markedly faster surplus decay than the standard SG, achieving substantially smaller errors for a given number of degrees of freedom. Overall, these results show that HO-SG provides the best trade-off between accuracy and computational cost and offers strong numerical evidence consistent with the fourth-order convergence claims established theoretically in this work.

\begin{figure}[ht]
  \centering
  \includegraphics[width=1.1\linewidth]{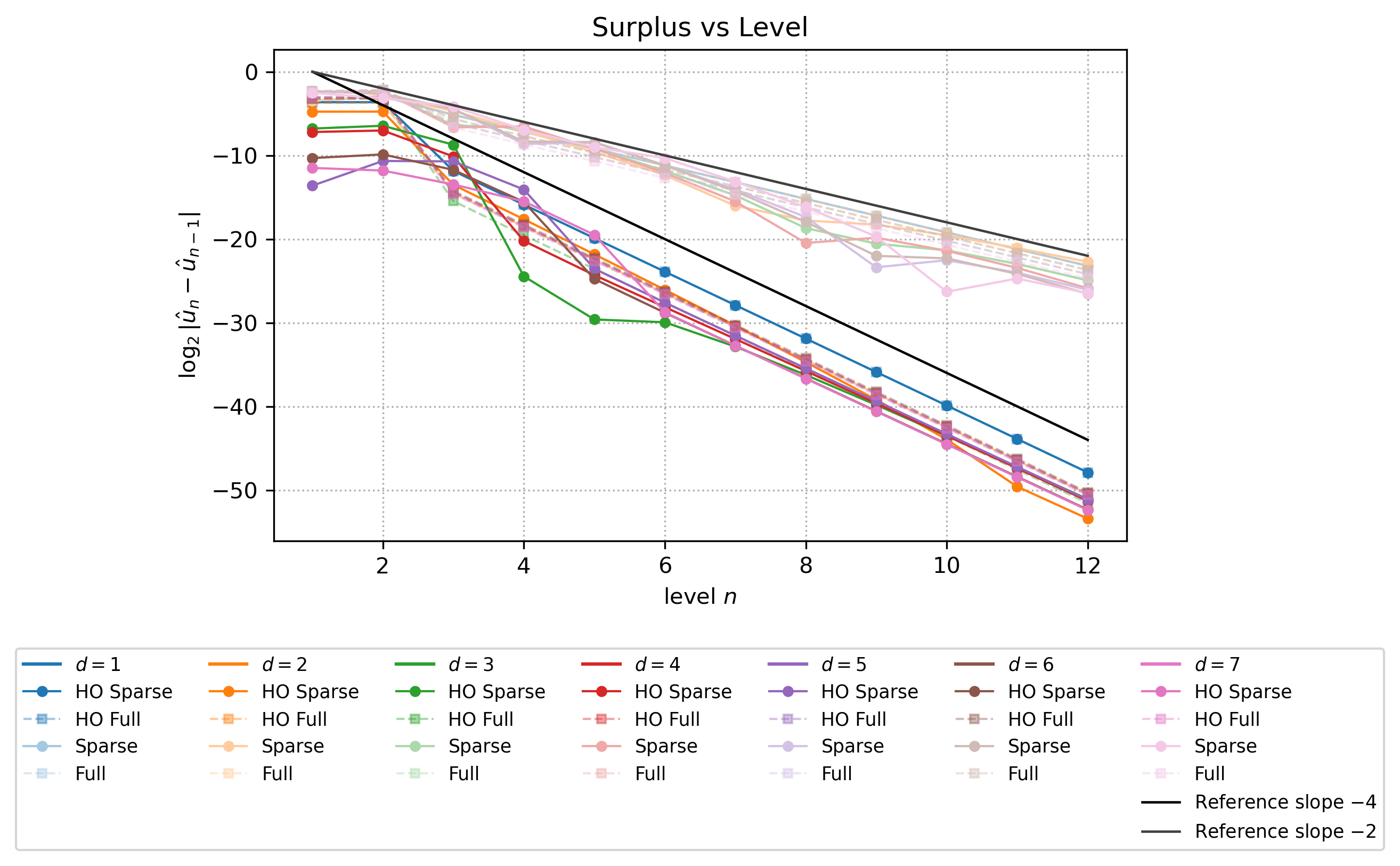}
  \caption{Surplus vs.\ level \(n\) for HO‐SG, SG, HO‐FG, and FG for $1\le d\le 7$.}
  \label{dims_surplus_level}
\end{figure}

\begin{figure}[ht]
  \centering
  \includegraphics[width=1.1\linewidth]{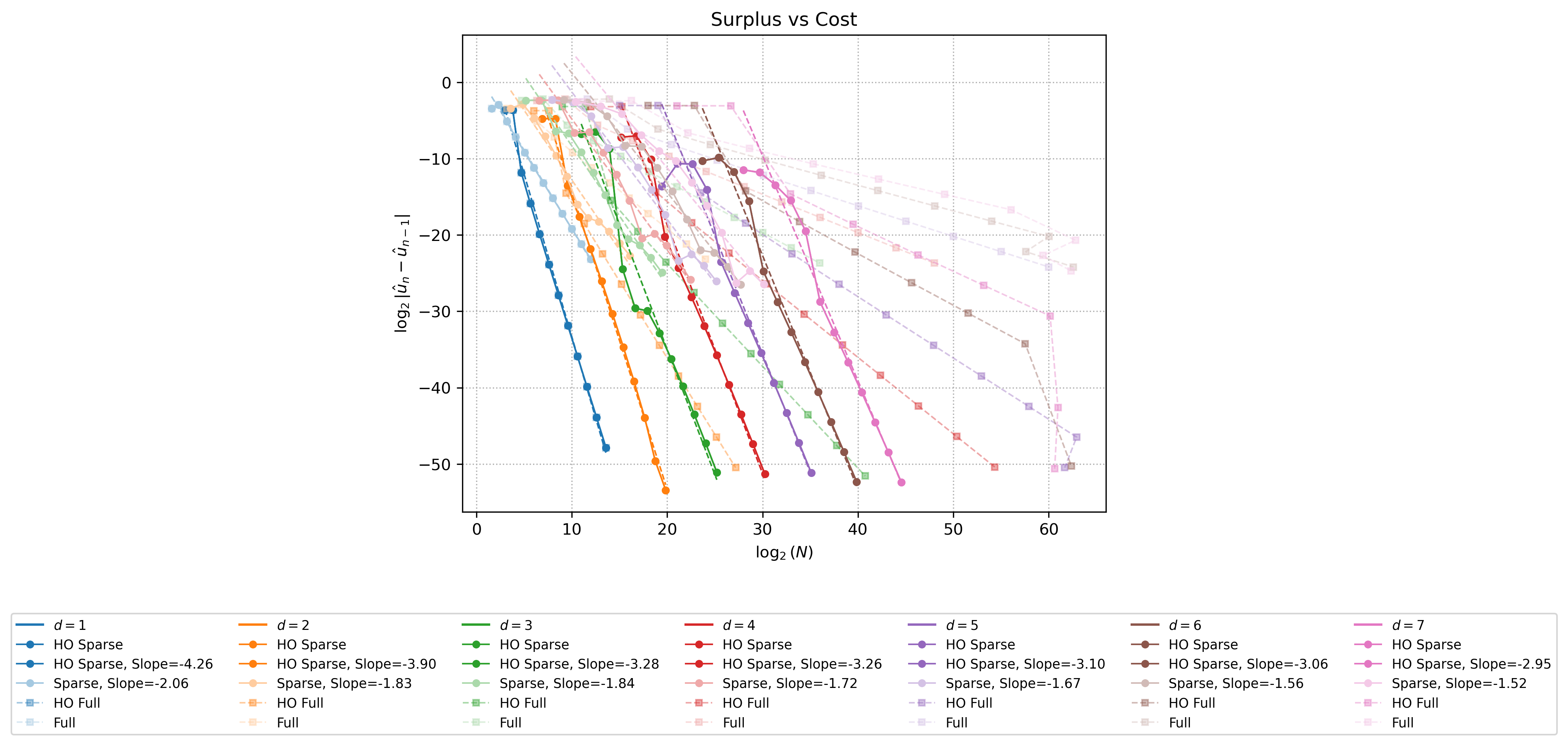}
  \caption{Surplus vs.\ number of grid points \(N\) for HO‐SG, SG, HO‐FG, and FG for $1\le d\le 7$.
  {\color{red}}}
  \label{dims_surplus_cost}
\end{figure}

\pagenumbering{gobble}


\begin{appendices}
\addcontentsline{toc}{section}{Appendices}



\section{Two Dimensional Error Expansion}
\label{proof-2d}

In this section, we give an informal proof of Theorem~\ref{mainresult} for the two-dimensional case \(d=2\) in order to illustrate the main ideas underlying the general argument. In particular, we establish the error expansion
\[
u - u_h = h_1^2 \beta_1(h_2) + h_2^2 \beta_2(h_1) + h_1^4 \gamma_1(h_1) + h_2^4 \gamma_2(h_2) + h_1^4 h_2^4 \gamma_{12}(h_1,h_2),
\]
which already isolates the core mechanisms responsible for higher-order error cancellation while avoiding the more cumbersome notation required in higher dimensions. The general \(d\)-dimensional case, treated rigorously in Section~5, builds on the same principles and extends the argument by accounting for the additional mixed terms that arise when \(d>2\), together with a precise statement of the required assumptions.

The starting point is a consistency assumption of order two 
(see Assumption~\ref{assumption} later for the general case) of the form
\[
A_h u - f = h_1^2 \sigma_1(h_1) + h_2^2 \sigma_2(h_2) + h_1^2 h_2^2 \sigma_{12}(h_1, h_2)
\]
where $\sigma_1$, $\sigma_2$ and $\sigma_{12}$ are even functions of the mesh sizes and so can be further expanded as 
\[\sigma_i(h_i) = \beta_i + h_i^2 \gamma_i^{(0)}(h_i), \qquad i = 1,2, \]
\[ \sigma_{12}(h_1, h_2) = \beta_{12} + h_1^2 \tilde{\beta}_2^{(0)}(h_1) + h_2^2 \tilde{\beta}_1^{(0)}(h_2) + h_1^2 h_2^2 \gamma_{12}^{(0)}(h_1, h_2) \]
so that the truncation error takes the more convenient form 

\begin{align*}
    A_h u - f &= h_1^2 \beta_1 + h_2^2 \beta_2 + h_1^2 h_2^2 \left[\beta_{12} + h_1^2\tilde{\beta}_2^{(0)}(h_1) \right] + h_1^2 h_2^4 \tilde{\beta}_1^{(0)}(h_2) \\
    & \qquad + h_1^4 \gamma_1^{(0)}(h_1) + h_2^4 \gamma_2^{(0)}(h_2) + h_1^4 h_2^4 \gamma_{12}^{(0)}(h_1, h_2) \\
    &= h_1^2 \beta_1 + h_2^2 \beta_2 + h_1^2 h_2^2 \beta_2^{(0)}(h_1) + h_1^2 h_2^2 \beta_1^{(0)}(h_2) \\
    & \qquad + h_1^4 \gamma_1^{(0)}(h_1) + h_2^4 \gamma_2^{(0)}(h_2) + h_1^4 h_2^4 \gamma_{12}^{(0)}(h_1, h_2) \\
\end{align*}
As a first step, in order to cancel the constant contributions $\beta_1$ and $\beta_2$ 
from the right-hand side, we introduce the following semi-discrete problems:
$$
\begin{cases}
    A_h^{(1)} \beta_2^{(1)}(h_1) = \beta_2 \\
    A_h^{(2)} \beta_1^{(1)}(h_2) = \beta_1
\end{cases}
$$
with respective truncation errors given by
\begin{align*}
    & \left[A_h^{(1)} - A_h \right] \beta_2^{(1)}(h_1) = h_2^2 \bar{\sigma}_{12}(h_1,h_2) = h_2^2 \left[ \tilde{\sigma}_{12} + h_1^2 \tilde{\sigma}_1(h_1) + h_2^2 \tilde{\sigma}_2(h_2) + h_1^2h_2^2 \tilde{\sigma}_{12}(h_1, h_2)\right]\\
    & \quad = h_2^2 \left[ \tilde{\sigma}_{12} + h_1^2 \left(\hat{\sigma}_1 + h_1^2 \hat{\sigma}_1(h_1) \right) + h_2^2 \tilde{\sigma}_2(h_2) + h_1^2h_2^2 \left(\sigma_{12} + h_1^2 \sigma_1(h_1) + h_2^2 \sigma_2(h_2) + h_1^2h_2^2 \sigma_{12}(h_1, h_2)\right) \right]
\end{align*}
where the successive expansions rely on Assumption~\ref{assumption}, and similar for 1 and 2 swapped.

Multiplying by $h_2^2$ and applying the same argument to $\beta_1$ yields
\[
    h_2^2 \left[ A_h^{(1)} - A_h \right] \beta_2^{(1)}(h_1) = h_2^4 \bar{\gamma}_2(h_2) + h_1^2 h_2^4 \delta_1(h_2) + h_1^4 h_2^4 \bar{\gamma}_{12} (h_1, h_2)
\]
and
\[
    h_1^2 \left[ A_h^{(2)} - A_h \right] \beta_1^{(1)}(h_2) = h_1^4 \bar{\gamma}_1(h_1) + h_2^2 h_1^4 \delta_2(h_1) + h_1^4 h_2^4 \bar{\gamma}_{21} (h_1, h_2). 
\]

Then 
\begin{align*}
    A_h & \left[ u - h_1^2 \beta_1^{(1)}(h_2) - h_2^2 \beta_2^{(1)}(h_1)\right]- f  \\
    = & h_1^4 \gamma_1^{(1)}(h_1) + h_2^4 \gamma_2^{(1)}(h_2) + h_1^4 h_2^4 \gamma_{12}^{(1)}(h_1, h_2) + h_1^2 h_2^2 \beta_1^{(0)}(h_2) + h_1^2 h_2^2 \beta_2^{(0)}(h_1) \\
    & \quad + h_1^2 h_2^4 \delta_1^{(0)}(h_2) + h_2^2 h_1^4 \delta_2^{(0)}(h_1) 
\end{align*}

Now consider the semi-discrete problems
$$
\begin{cases}
    A_h^{(1)} \gamma_1^{(2)}(h_1) = \gamma_1^{(1)}(h_1) \\
    A_h^{(2)} \gamma_2^{(2)}(h_2) = \gamma_2^{(1)}(h_2)
\end{cases}
$$
with truncation errors given by 
\begin{align*}
    \left[ A_h^{(1)} - A_h \right]\gamma_1^{(2)}(h_1) & = h_2^2 \bar{\sigma}_{12}(h_1,h_2)\\
    & =  h_2^2\left[ \tilde{\sigma}_{12} + h_1^2 \tilde{\sigma}_1^{(1)}(h_1) + h_2^2 \tilde{\sigma}_2(h_2) + h_1^2h_2^2 \tilde{\sigma}_{12}(h_1, h_2)\right]
\end{align*}
with 
\[
    h_1^4 \left[ A_h^{(1)} - A_h \right]\gamma_1^{(2)}(h_1) = h_2^2 h_1^4 \tilde{\delta}_2(h_1) + h_1^4 h_2^4 \tilde{\gamma}_{12}(h_1, h_2)
\]
and similarly with $1$ and $2$ swapped. 
Then 
\begin{align*}
    A_h & \left[ u - h_1^2 \beta_1^{(1)}(h_2) - h_2^2 \beta_2^{(1)}(h_1) - h_1^4 \gamma_1^{(2)}(h_1) - h_2^4 \gamma_2^{(2)}(h_2) \right]- f \\
    & =  h_1^4 h_2^4 \gamma_{12}^{(2)}(h_1, h_2) + h_1^2 h_2^2 \beta_1^{(0)}(h_2) + h_1^2 h_2^2 \beta_2^{(0)}(h_1) \\
    & \quad +  h_1^2 h_2^4 \delta_1^{(1)}(h_2)  +  h_2^2 h_1^4 \delta_2^{(1)}(h_1)
\end{align*}
To cancel the mixed terms $\beta_1^{(0)}(h_2)$ and $\beta_2^{(0)}(h_1)$ on the RHS, 
we next consider the semi–discrete problems
$$
\begin{cases}
    A_h^{(1)} \beta_2^{(2)}(h_1) = \beta_2^{(0)}(h_1)\\
    A_h^{(2)} \beta_1^{(2)}(h_2) = \beta_1^{(0)}(h_2)
\end{cases}
$$
with truncation errors given by
\begin{align*}
    \left[ A_h^{(1)} - A_h \right]\beta_2^{(2)}(h_1) & = h_2^2 \bar{\sigma}_{12}(h_1,h_2)\\
    & =  h_2^2\left[ \tilde{\sigma}_{12} + h_1^2 \tilde{\sigma}_1(h_1) + h_2^2 \tilde{\sigma}_2(h_2) + h_1^2h_2^2 \tilde{\sigma}_{12}(h_1, h_2)\right]
\end{align*}
with 
\[
h_1^2 h_2^2 \left[ A_h^{(1)} - A_h \right]\beta_2^{(2)}(h_1) = h_2^2 h_1^4 \tilde{\delta}_1(h_2) + h_1^4 h_2^4 \tilde{\gamma}_{12}(h_1, h_2)
\]
and  
\[
h_1^2 h_2^2 \left[ A_h^{(2)} - A_h \right]\beta_1^{(2)}(h_2) = h_1^2 h_2^4 \tilde{\delta}_2(h_1) + h_1^4 h_2^4 \tilde{\gamma}_{21}(h_1, h_2). 
\]
Then 
\begin{align*}
    A_h & \left[ u - h_1^2 \beta_1^{(1)}(h_2) - h_2^2 \beta_2^{(1)}(h_1) - h_1^4 \gamma_1^{(2)}(h_1) - h_2^4 \gamma_2^{(2)}(h_2) - h_1^2 h_2^2 \beta_1^{(2)}(h_2) - h_1^2h_2^2 \beta_2^{(2)}(h_1)  \right]- f \\
    = & \ h_1^4 h_2^2 \gamma_{12}^{(3)} (h_1, h_2) + h_1^2 h_2^4 \delta_1^{(2)}(h_2) + h_2^2 h_1^4 \delta_2^{(2)}(h_1).
\end{align*}

Finally, consider the semi-discrete problems
$$
\begin{cases}
    A_h^{(1)} \delta_2^{(3)}(h_1) = \delta_2^{(2)}(h_1) \\
    A_h^{(2)} \delta_1^{(3)}(h_2) = \delta_1^{(2)}(h_2)
\end{cases}
$$
with truncation errors given by 
$$
\begin{cases}
    \left[ A_h^{(1)} - A_h \right] \delta_2^{(3)} (h_1) = h_2^2 \bar{\sigma}_{12}(h_1, h_2).  \\
    \left[ A_h^{(2)} - A_h \right] \delta_1^{(3)} (h_2) = h_1^2 \bar{\sigma}_{21}(h_1, h_2). 
\end{cases}
$$
For notational simplicity, we relabel the resulting remainder terms as \(\bar{\sigma}_{12} = \tilde{\gamma}_{12}\) and \(\bar{\sigma}_{21} = \tilde{\gamma}_{21}\), and write
$$
\begin{cases}
    h_2^2 h_1^4 \left[ A_h^{(1)} - A_h \right] \delta_2^{(3)} (h_1) = h_1^4 h_2^4 \tilde{\gamma}_{12}(h_1, h_2). \\
    h_1^2 h_2^4 \left[ A_h^{(2)} - A_h \right] \delta_1^{(3)} (h_2) = h_1^4 h_2^4 \tilde{\gamma}_{21}(h_1, h_2).
\end{cases}
$$
Subtracting $\delta_1^{(3)}$ and  $\delta_2^{(3)}$, finally yields
\begin{align*}
    A_h & [ u - h_1^2 \beta_1^{(1)}(h_2) - h_2^2 \beta_2^{(1)}(h_1) - h_1^4 \gamma_1^{(2)}(h_1) - h_2^4 \gamma_2^{(2)}(h_2) \\
    & \qquad - h_1^2 h_2^2 \beta_1^{(2)}(h_2) - h_1^2h_2^2 \beta_2^{(2)}(h_1) - h_2^2 h_1^4 \delta_2^{(3)}(h_1) - h_1^2 h_2^4 \delta_1^{(3)}(h_2) ]- f \\
   & = h_1^4 h_2^4 \gamma_{12}^{(4)} (h_1, h_2).
\end{align*}

This proves the desired result in two dimensions. 

The general \(d\)-dimensional case, treated rigorously in the next section, builds on the same principles, provides rigorous justification for the error expansions used, and accounts for the additional mixed terms that appear when \(d>2\).

\section{Proofs from Section \ref{sec:Chapter 5} }\label{Appendix}
\begin{lem}
Let $\{ i_1, \ldots, i_n \} \subseteq \{k_1, \ldots, k_p\} \subseteq \{1, \ldots, d\}$ and $\{j_1, \ldots, j_m\} \cap \{k_1, \ldots, k_p\} = \emptyset$. 

Let 
\(
\phi(\cdot; h_{i_1}, \ldots, h_{i_n})
\)
be any sufficiently regular function, smooth in its spatial variables and 
smooth and even in each mesh-size component.  
Suppose that the semi-discrete problem
\[
A_h^{(k_1, \ldots, k_p)} \psi(\cdot; h_{k_1}, \ldots, h_{k_p})
    = \phi(\cdot; h_{k_1}, \ldots, h_{k_p})
\]
admits a solution $\psi$ satisfying the regularity requirements of 
Assumption~\ref{assumption}.  
Then the corresponding solution 
\(
\psi(\cdot; h_{i_1}, \ldots, h_{i_n}) \in C^\infty(I^d)
\)
inherits the same smoothness and evenness properties in the mesh sizes, and 
the truncation error satisfies:
\begin{align*}
    & h_{j_1}^2 \cdots h_{j_m}^2 \, h_{i_1}^4 \cdots h_{i_n}^4 \, \big[ A_h^{(k_1, \ldots, k_p)} - A_h \big] \psi(h_{k_1}, \ldots, h_{k_p})  \\
    & = \sum_{q = m+n+1}^d \sum_{ \substack{ \{l_1, \ldots, l_q\} \\ \subset \{1, \ldots, d\} }} 
    h_{l_1}^4 \cdots h_{l_q}^4 \, \gamma_{l_1, \ldots, l_q}(\cdot; h_{l_1}, \ldots, h_{l_q}) \\
    & \quad + \sum_{q = m+n+1}^d \sum_{ \substack{ \{l\} \sqcup \{l_2, \ldots, l_q\} \\ \subset \{1, \ldots, d\} }} 
    h_l^2 h_{l_2}^2 \cdots h_{l_q}^2 \, \beta_{l l_2 \ldots l_q}(h_{l_2}, \ldots, h_{l_q}) \\
    & \quad + \sum_{q = n+1}^{m+n} \sum_{ \substack{ \{s_1, \ldots, s_q\} \sqcup \{r_1, \ldots, r_{m+n-q}\} \\ \subset \{1, \ldots, d\} }} 
    h_{r_1}^2 \cdots h_{r_{m+n-q}}^2 \, h_{s_1}^4 \cdots h_{s_q}^4 \, 
    \sigma_{s_1, \ldots, s_q}^{r_1, \ldots, r_{m+n-q}}(h_{s_1}, \ldots, h_{s_q}).
\end{align*}
Here, $\gamma_{l_1, \ldots, l_q}$, $\beta_{l l_2 \ldots l_q}$, and $\sigma_{s_1, \ldots, s_q}^{r_1, \ldots, r_{m+n-q}}$ are smooth functions that are also smooth and even in the mesh-size variables.
\end{lem}

\begin{proof}
    By Assumtion \ref{assumption} and regularity of $\psi$, 
    \[
    \left[ A_h^{(k_1, \ldots, k_p)} - A_h\right] \psi(h_{k_1}, \ldots, h_{k_p}) = \sum_{q=1}^{d-p} \sum_{\substack{ \{l_1, \ldots, l_q\} \\ \subset \{1, \ldots, d\}-\{k_1, \ldots, k_p\} }} h_{l_1}^2 \dots h_{l_q}^2 \phi_{l_1, \ldots, l_1}(h_{l_1}, \ldots, h_{l_q}, h_{k_1}, \dots, h_{k_p})
    \]
    where the coefficient functions $\phi_{l_1, \ldots, l_q}$ are smooth even functions of the meshsizes. 

    By Lemma \ref{evenexp} we can further expand each $\phi_{l_1, \ldots, l_q}$ with respect to $(h_{k_1}, \ldots, h_{k_p})$ as 
    \[
    \phi_{l_1, \ldots, l_q}(h_{l_1}, \dots, h_{l_q}, h_{k_1}, \dots, h_{k_p}) = \sum_{ w=0 }^p \sum_{ \substack{\{g_1, \ldots, g_w\} \\ \subset \{k_1, \ldots, k_p\} } } h_{g_1}^2 \dots h_{g_w}^2 \delta_{g_1, \ldots, g_w}^{l_1, \ldots, l_q}(h_{l_1}, \dots, h_{l_q}, h_{g_1}, \dots, h_{g_w}). 
    \]

    This then gives the error expansion 
    \begin{align}\label{lemap}
        & h_{i_1}^4 \dots h_{i_n}^4 h_{j_1}^2 \dots h_{j_m}^2 \left[ A_h^{(k_1, \ldots, k_p) }- A_h\right] \psi(h_{k_1}, \dots, h_{k_p}) =  \\
        = & h_{i_1}^4 \dots h_{i_n}^4 h_{j_1}^2 \dots h_{j_m}^2 \sum_{\substack{ 1 \leq q \leq d-p \\ 0 \leq w \leq p  }} \sum_{\substack{ \{g_1, \ldots, g_w\} \subset \{k_1, \dots, k_p\} \\ \{l_1, \dots ,l_q\} \cap \{k_1, \dots, k_p\} = \emptyset }} h_{l_1}^2 \dots h_{l_q}^2 h_{g_1}^2 \dots h_{g_w}^2 \delta_{g_1, \dots, g_w}^{l_1, \dots, l_q}(h_{l_1}, \dots, h_{l_q}, h_{g_1}, \dots, h_{g_w}). \notag
    \end{align}
    
The proof of the lemma now boils down to identifying each summand in the above expression (\ref{lemap}) as either
\begin{enumerate}
    \item a pure fourth order term $h_{l_1}^4\dots h_{l_q}^4 \gamma_{l_1\dots l_q}(h_{l_1}, \dots, h_{l_q})$ for some $q \geq m+n+1$. 
    \item a mixed second order term of the form $h_l^2 h_{l_2}^2 \dots h_{l_q}^2 \beta_{ll_2\dots l_q}(h_{l_2}, \dots, h_{l_q})$ for some $q \geq m+n+1$. 
    \item a term of the form $ h_{r_1}^2 \dots h_{r_{m+n-q}}^2 h_{s_1}^4 \dots h_{s_q}^4 \sigma_{s_1\dots s_q}^{r_1\dots r_{m+n-q}}( h_{s_1}, \dots, h_{s_q} )$ for some $n+1 \leq q \leq m+n$. 
\end{enumerate}

The detailed case-by-case grouping below carries out this classification and thereby completes the proof of the lemma. In the following, it is important to note that $\{g_1, \dots, g_w\} \cup \{i_1, \dots, i_n\} \subseteq \{k_1, \dots, k_p\}$ and $ \big( \{l_1, \dots, l_q\} \cup \{j_1, \dots, j_m\} \big)\cap \{k_1, \dots, k_p\} = \emptyset$. 
\begin{itemize}
    \item The term with $\{i_1, \dots, i_n\} = \{g_1, \dots, g_w\}$ and  $\{j_1, \dots, j_m\} = \{l_1, \dots, l_q\}$ falls into the last category (3) with $q = m+n$: 
    \begin{align*}
        & h_{i_1}^4 \dots h_{i_n}^4 h_{j_1}^2 \dots h_{j_m}^2 h_{g_1}^2 \dots h_{g_w}^2 h_{l_1}^2 \dots h_{l_q}^2 \phi_{g_1\dots g_w}^{l_1\dots l_q}(h_{l_1}, \dots, h_{l_q}, h_{g_1}, \dots, h_{g_w}) \\ 
        = & h_{i_1}^4 \dots h_{i_n}^4 h_{j_1}^4 \dots h_{j_m}^4 \hat{\phi}_{i_1 \dots i_n}^{j_1 \dots j_m}(h_{i_1}, \dots, h_{i_n}, h_{j_1}, \dots, h_{j_m}). 
    \end{align*}
     
    \item For the terms with $\{j_1, \dots j_m\} \subseteq \{l_1, \dots, l_q\}$, let $\{i_1, \dots, i_n\} \cup \{g_1, \dots, g_w\} = \{u_1, \dots, u_a\}$ with $a \geq n$. WLOG we may assume that one of $a > n$ or $ q > m$, so that $a+q \geq m+n+1$ (otherwise we fall back to the previous case). Then 
    \begin{align*}
           & h_{i_1}^4 \dots h_{i_n}^4 h_{j_1}^2 \dots h_{j_m}^2 h_{g_1}^2 \dots h_{g_w}^2 h_{l_1}^2 \dots h_{l_q}^2 \phi_{g_1\dots g_w}^{l_1\dots l_q}(h_{l_1}, \dots, h_{l_q}, h_{g_1}, \dots, h_{g_w}) \\
           = & h_{u_1}^2 \dots h_{u_a}^2 h_{l_1}^2 \dots h_{l_q}^2 \hat{\phi}_{g_1 \dots g_w}^{l_1\dots l_q}(h_{u_1}, \dots, h_{u_a}, h_{l_1}, \dots, h_{l_q}) \\
           \eqqcolon &  h_{d_1}^2 \cdots h_{d_r}^2 \hat{\phi}_{g_1\dots g_w}^{l_1\dots l_q}(h_{d_1}, \dots, h_{d_r}) \\
           = & h_{d_1}^4 \cdots h_{d_r}^4 \sigma_{d_1\dots d_r}(h_{d_1}, \dots, h_{d_r}) \\
           & + \sum_{s=0}^{r-1}\sum_{ \substack{ \{c_1, \dots, c_s\} \sqcup \{ b_1, \dots, b_{r-s} \} \\ = \{d_1, \dots, d_r\} }} h_{b_1}^2 \cdots h_{b_{r-s}}^2 h_{d_1}^4 \dots h_{d_s}^4 \sigma_{d_1\dots d_s}(h_{d_1}, \dots, h_{d_s}) \\
           = & h_{d_1}^4 \cdots h_{d_r}^4 \sigma_{d_1\dots d_r}(h_{d_1}, \dots, h_{d_r}) \\
           & + \sum_{s=0}^{r-1}\sum_{ \substack{ \{c_1, \dots, c_s\} \sqcup \{ b_1, \dots, b_{r-s} \} \\ = \{d_1, \dots, d_r\} }} h_{b_1}^2 \cdots h_{b_{r-s}}^2 h_{d_1}^2 \dots h_{d_s}^2 \hat{\sigma}_{d_1\dots d_s}(h_{d_1}, \dots, h_{d_s})
    \end{align*}
    where $r = a+q\geq m+n+1$, $\{u_1,\dots, u_a\} \sqcup \{l_1, \dots, l_q\} = \{d_1, \dots, d_r\}$ and the term 
$\hat{\phi}_{g_1 \dots g_w}^{\,l_1 \dots l_q}$ has been expanded using 
Lemma~\ref{evenexp}.

These terms can be expressed as linear combination of a pure fourth order term (1) and mixed second order terms (2). 
    \item For the remaining terms, there exists $j \in \{j_1, \dots, j_m\}$ such that $j \notin \{l_1, \dots, l_q\}$. Up to relabeling, we may assume that $j = j_1$. First consider the case in which $ \{l_1, \dots, l_q\} \subseteq \{j_2, \dots, j_m\}$ and $\{g_1, \dots, g_w\} \subseteq \{i_1, \dots, i_n\}$ so that $ \{j, j_2, \dots, j_m\} \cup \{l_1, \dots, l_q\} \cup \{g_1, \dots, g_w\} \cup  \{i_1, \dots, i_n\}$ contains exactly $m+n$ distinct terms. Let $\{l_1, \dots, l_q\} \sqcup \{a_1, \dots, a_{m-1-q}\} = \{j_2, \dots, j_m\}$, then
    \begin{align*}
        & h_{i_1}^4 \dots h_{i_n}^4 h_{j_1}^2 \dots h_{j_m}^2 h_{g_1}^2 \dots h_{g_w}^2 h_{l_1}^2 \dots h_{l_q}^2 \phi_{g_1\dots g_w}^{l_1\dots l_q}(h_{l_1}, \dots, h_{l_q}, h_{g_1}, \dots, h_{g_w}) \\
        = & h_j^2 h_{a_1}^2 \dots h_{a_{m-1-q}}^2 h_{l_1}^4 \dots h_{l_q}^4 h_{i_1}^4 \dots h_{i_n}^4 \hat{\phi}(h_{l_1}, \dots, h_{l_q}, h_{i_1}, \dots, h_{i_n}). 
    \end{align*}
    Note that here $n+q \geq n+1$, so these terms fall under the last case (3). 
    \item Finally, for the remaining terms either there exists $l \in \{l_1, \dots, l_q\}$ such that $l \notin \{j_2, \dots, j_m\}$ or there is $g \in \{g_1, \dots, g_w\}$ such that $g \notin \{i_1, \dots, i_n\}$. Then $ \{j, j_2, \dots, j_m\} \cup \{l_1, \dots, l_q\} \cup \{g_1, \dots, g_w\} \cup \{i_1, \dots, i_n\} \eqqcolon \{j\} \sqcup \{a_1, \dots, a_r \}$ contains at least $ m+n+1$ distinct terms. Then 
    \begin{align*}
        & h_{i_1}^4 \dots h_{i_n}^4 h_{j_1}^2 \dots h_{j_m}^2 h_{g_1}^2 \dots h_{g_w}^2 h_{l_1}^2 \dots h_{l_q}^2 \phi_{g_1\dots g_w}^{l_1\dots l_q}(h_{l_1}, \dots, h_{l_q}, h_{g_1}, \dots, h_{g_w}) \\
        = & h_j^2 h_{a_1}^2 \dots h_{a_r}^2  \hat{\phi}(h_{a_1}, \dots, h_{a_r}), 
    \end{align*}
    which is of the form of the second–order terms in (2).
\end{itemize}
\end{proof}

\end{appendices}


\begin{thebibliography}{99}


\bibitem{bachmayr2023low} Bachmayr, M., \textit{Low-rank tensor methods for partial differential equations}, Acta Numerica, 32, 1--121, 2023.

\bibitem{bungartz1994extrapolation} Bungartz, H.; Griebel, M.; R{\"u}de, U., \textit{Extrapolation, combination, and sparse grid techniques for elliptic boundary value problems}, Computer Methods in Applied Mechanics and Engineering, 116(1-4), 243--252, 1994.

\bibitem{Bungartz1992} Bungartz, H.~J., \textit{D\"unne Gitter und deren Anwendung bei der adaptiven L\"osung der dreidimensionalen Poisson-Gleichung}, 1992.

\bibitem{Bungartz1998} Bungartz, H.~J., \textit{Finite elements of higher order on sparse grids}, 1998. Habilitationsschrift.

\bibitem{BungartzGriebelRoeschkeZenger1994} Bungartz, H.~J.; Griebel, M.; Röschke, D.; Zenger, C., \textit{Pointwise convergence of the combination technique for the Laplace equation}, East-West Journal of Numerical Mathematics, 2, 21--45, 1994.

\bibitem{BungartzGriebel2004} Bungartz, H.~J.; Griebel, M., \textit{Sparse Grids}, Acta Numerica, 13, 147--269, 2004.

\bibitem{griebel2014convergence} Griebel, M.; Harbrecht, H., \textit{On the convergence of the combination technique}, In Sparse Grids and Applications-Munich 2012, pp.~55--74, Springer, 2014.

\bibitem{GriebelSchneiderZenger1992} Griebel, M.; Schneider, M.; Zenger, C., \textit{A combination technique for the solution of sparse grid problems}, In Iterative Methods in Linear Algebra, IMACS, Elsevier, North Holland, 1992.

\bibitem{han2018solving} Han, J.; Jentzen, A.; E, W., \textit{Solving high-dimensional partial differential equations using deep learning}, Proceedings of the National Academy of Sciences, 115(34), 8505--8510, 2018.

\bibitem{harding2014combination} Harding, B., \textit{Combination technique coefficients via error splittings}, The Proceedings of ANZIAM, 56, C355--C368, 2014.

\bibitem{harding2016adaptive} Harding, B., \textit{Adaptive sparse grids and extrapolation techniques}, In Sparse Grids and Applications-Stuttgart 2014, pp.~79--102, Springer, 2016.

\bibitem{hegland2007combination} Hegland, M.; Garcke, J.; Challis, V., \textit{The combination technique and some generalisations}, Linear Algebra and its Applications, 420(2-3), 249--275, 2007.

\bibitem{hegland2016recent} Hegland, M.; Harding, B.; Kowitz, C.; Pfl{\"u}ger, D.; Strazdins, P., \textit{Recent developments in the theory and application of the sparse grid combination technique}, Software for Exascale Computing-SPPEXA 2013-2015, 143--163, 2016.

\bibitem{hendricks2016high} Hendricks, C.; Ehrhardt, M.; G{\"u}nther, M., \textit{High-order {ADI} schemes for diffusion equations with mixed derivatives in the combination technique}, Applied Numerical Mathematics, 101, 36--52, 2016.

\bibitem{hendricks2017error} Hendricks, C.; Ehrhardt, M.; G{\"u}nther, M., \textit{Error splitting preservation for high order finite difference schemes in the combination technique}, Numerical Mathematics: Theory, Methods and Applications, 10(3), 689--710, 2017.

\bibitem{leentvaar2006pricing} Leentvaar, C.; Oosterlee, C.~W., \textit{Pricing multi-asset options with sparse grids and fourth order finite differences}, In Numerical Mathematics and Advanced Applications: Proceedings of ENUMATH 2005, the 6th European Conference on Numerical Mathematics and Advanced Applications Santiago de Compostela, Spain, July 2005, pp.~975--983, Springer, 2006.

\bibitem{vonPetersdorffSchwab2004} Petersdorff, T.~V.; Schwab, C., \textit{Numerical solution of parabolic equations in high dimensions}, ESAIM: Mathematical Modelling and Numerical Analysis, 38(1), 93--127, 2004.

\bibitem{Pflaum1997} Pflaum, C., \textit{Convergence of the combination technique for second-order elliptic differential equations}, SIAM Journal on Numerical Analysis, 34(6), 2431--2455, 1997.

\bibitem{PflaumZhou1999} Pflaum, C.; Zhou, A., \textit{Error analysis of the combination technique}, Numerische Mathematik, 84, 327--350, 1999.

\bibitem{raissi2019physics} Raissi, M.; Perdikaris, P.; Karniadakis, G.~E., \textit{Physics-informed neural networks: A deep learning framework for solving forward and inverse problems involving nonlinear partial differential equations}, Journal of Computational physics, 378, 686--707, 2019.

\bibitem{Reisinger2004} Reisinger, C., \textit{Efficient Numerical Methods for High-Dimensional Parabolic Equations and Applications in Option Pricing}, 2004.

\bibitem{reisinger2013analysis} Reisinger, C., \textit{Analysis of linear difference schemes in the sparse grid combination technique}, IMA Journal of Numerical Analysis, 33(2), 544--581, 2013.

\bibitem{ReisingerWittum2007} Reisinger, C.; Wittum, G., \textit{Efficient hierarchical approximation of high-dimensional option pricing problems}, SIAM Journal on Scientific Computing, 29(1), 440--458, 2007.

\bibitem{richardson1927viii} Richardson, L.~F.; Gaunt, J.~A., \textit{The deferred approach to the limit}, Philosophical Transactions of the Royal Society of London. Series A, containing papers of a mathematical or physical character, 226(636-646), 299--361, 1927.

\bibitem{schuller1985efficient} Sch{\"u}ller, A.; Lin, Q., \textit{Efficient high order algorithms for elliptic boundary value problems combining full multigrid techniques and extrapolation methods}, 1985.

\bibitem{SchwabEtAl2008} Schwab, C.; S\"uli, E.; Todor, R.~A., \textit{Sparse finite element approximation of high-dimensional transport-dominated diffusion problems}, ESAIM: Mathematical Modelling and Numerical Analysis, 42(5), 777--820, 2008.

\bibitem{sirignano2018dgm} Sirignano, J.; Spiliopoulos, K., \textit{DGM: A deep learning algorithm for solving partial differential equations}, Journal of computational physics, 375, 1339--1364, 2018.

\bibitem{Smolyak1963} Smolyak, S., \textit{Quadrature and interpolation formulas for tensor products of certain classes of functions}, Soviet Math. Dokl., 4, 240--243, 1963.

\bibitem{Smolyak1963} Smolyak, S.~A., \textit{Quadrature and interpolation formulas for tensor products of certain classes of functions}, Doklady Akademii Nauk SSSR, 148, 1042--1045, 1962. English translation in Soviet Math. Dokl. {\bf4}:240--243, 1963.

\bibitem{Zenger1990} Zenger, C., \textit{Sparse grids}, In Parallel Algorithms for Partial Differential Equations, 1990. Proceedings of the Sixth GAMM-Seminar.

\end{thebibliography}
\end{document}